\documentclass[12pt,amssymb,verbatim]{amsart}

\usepackage{amsfonts,latexsym,geometry,color,hyperref,ulem}
\usepackage{amssymb,graphics,graphicx}
\usepackage{amsfonts,latexsym,geometry}
\usepackage{amssymb,graphics,graphicx}
\usepackage{bbm}
\usepackage{amscd}
\usepackage{amssymb}
\usepackage{amsthm}
\usepackage[utf8]{inputenc} 
\usepackage[T1]{fontenc}
\usepackage{amsmath}
\usepackage{amsfonts}
\usepackage{amssymb}
\usepackage{graphicx} 
\usepackage{enumerate}
\newcommand\supp{\mathrm{supp}}
\newtheorem{theoreme}{Theorem}[section] %
\newtheorem{proposition}[theoreme]{Proposition} %
\newtheorem{corollary}[theoreme]{Corollary} %
\newtheorem{lemme}[theoreme]{Lemma} %
\newtheorem{definition}{Definition}[section] %
\newtheorem{remark}[theoreme]{Remark} %

\newcommand\sk{\smallskip}

\newcommand\R{\mathbb{R}} 
 
\newcommand\dimm{\underline{\dim}_H}

\RequirePackage{yhmath} 
\renewcommand\widering[1]{\ring{#1}}

\author{E. Daviaud, Uliège}

\begin{document}
\title{Random covering by rectangles  on  self-similar carpets}
\maketitle

\begin{abstract}
In this article, given a base-$b$ self-similar set $K$, we study the random covering of $K$ by horizontal or vertical rectangles, with respect to the Alfhors-regular measure on $K$, and the rectangular shrinking target problem on $K$.   
\end{abstract}

\section{Introduction}

The metric approximation theory aims at estimating the dimension of sets of elements which are approximable at a certain rate by a sequence of particular points of interest. More precisely, given $d\in \mathbb{N}$, $(x_n)_{n\in\mathbb{N}}\in\Big(\mathbb{R}^{d}\Big)^{\mathbb{N}}$ and $\psi:\mathbb{N}\to \mathbb{R}_+$ a mapping satisfying $\lim_{n\to +\infty}\psi(n)=0$, the set $E_{\psi}$ of elements approximable at rate  $\psi$ by the sequence $(x_n)_{n\in \mathbb{N}}$ is defined as  $$\dim_H E_{\psi}:=\left\{y\in\mathbb{R}^d : \ \vert\vert y -x_n\vert\vert_{\infty} \leq \psi(n)\text{ i.o. }\right\},$$
where $i.o.$  ($``$ infinitely often $"$) means that the inequality holds for infinitely many $n$. Such problems were originally born in Diophantine approximation, where one studies sets of real numbers or vectors approximable by rational numbers or vectors at a given speed rate. These questions are also natural in multifractal analysis, as for many mappings, the regularity at a given point depends on its rate of approximation by specific dyadic numbers or rational numbers (see \cite{JaffRiemann} for instance) and in dynamical and random approximation, as, for instance, given an ergodic system $(T,\mu),$  the local dimension at a given point $x$  depends on its approximation rate by typical $\mu$ orbits, provided that $\mu$ mixes sufficiently fast (see \cite{Galadim,Edergo}). In particular the theory of random approximation as raised many interests these last 30 years and has known recently many developments. Given $\mu \in\mathcal{M}(\mathbb{R}^d)$ a probability measure and $(X_n)_{n\in\mathbb{N}}$ an i.i.d. sequence of random variable of law $\mu$ it was established in \cite{JJMS} that, denoting $\overline{\dim}_H \mu$ the upper Hausdorff dimension of $\mu$ (see Definition \ref{dim}), for any $\delta \geq \frac{1}{\overline{\dim}_H \mu},$ almost surely, one has  
\begin{equation}
\dim_H \left\{y \in\mathbb{R}^d : \ \vert\vert y-X_n \vert\vert_{\infty}\leq \frac{1}{n^{\delta}}\text{ i.o. }\right\}:=\limsup_{n\to +\infty}B(X_n,\frac{1}{n^{\delta}})=\frac{1}{\delta} .
\end{equation}
This result was later on generalized in \cite{JarvSeur}, showing that no simple formula (depending on geometric quantity related to $\mu$ one usually considers) holds in general in the case $\delta<\frac{1}{\overline{\dim}_H \mu}$, solving a conjecture of Eckström and Persson stated in \cite{EP}. The question of random approximation by other shapes than ball has also been considered. An important result regarding this topic was established in \cite{FengJ2Suo} in the case of measure which are not purely singular. In particular, the authors proved that, given $(X_n)_{n\in\mathbb{N}}$  i.i.d. of law the Lebesgue measure $\mathcal{L}^d$ on $\mathbb{T}^d$ and $(O_n)_{n\in\mathbb{N}}$   a sequence of open sets satisfying $\vert O_n \vert \to 0,$  one has almost surely,
\begin{equation}
\label{ThmFJJS}
\dim_H \limsup_{n\to+\infty}\Big(X_n +O_n \Big)=\inf\left\{t: \sum_{n\geq 1}\mathcal{H}^t_{\infty}(O_n)<+\infty\right\},
\end{equation}

where $\mathcal{H}^t_{\infty}(O_n)$ denotes the Hausdorff content of dimension $t$ of $O_n$ (Definition \ref{hcont}).

In the case where the measure $\mu$ is singular, very little is known about the approximation by specific sequences of open sets. Of course, there is no reason, a priori, for a tractable formula to hold  as in \eqref{ThmFJJS}, if the sets $(O_n)_{n\in\mathbb{N}}$ do not enjoy special properties with respect to $\mu.$ In the present article, we study the random approximation by horizontal (or vertical) rectangles on a self-similar carpet. Let $b\in\mathbb{N}$ be an integer and let $K \subset [0,1]^2$ be a base $b$-missing digit set (see Definition \ref{Defmiss}), $\mu_0$ the Alfhors-regular measure on $K$ and $(X_n)_{n\in\mathbb{N}}$ an i.i.d. sequence of law $\mu_0.$ Let us fix also  $\frac{1}{\dim_H K}\leq \tau_1 \leq \tau_2$, $\tau=\frac{\tau_2}{\tau_1}$ and define $$W_{\tau_1 ,\tau_2}=\limsup_{n\to + \infty}\Big(X_n +(-\frac{1}{n^{\tau_1}},\frac{1}{n^{\tau_1}})\times (-\frac{1}{n^{\tau_2}},\frac{1}{n^{\tau_2}})\Big) $$

and for $\alpha \geq 0,$ write $$v_{\tau}(\alpha)=\dim_H K +(\tau-2)\alpha -(\tau-1)D_{\pi_2 \mu_0}(\alpha),$$
where $\pi_2 \mu_0$ denotes the projection of $\mu_0$ along the $y$-axis and $D_{\pi_2 \mu_0}$ its multifractal spectrum (Definition \ref{defmultifractal}). Let $\beta_{\tau_1,\tau_2}$ be the smallest solution (when well defined) of $v_{\tau}(\beta)=\frac{1}{\tau_1}.$ 

Then, there exists $\kappa_2 \geq 0$ such that, almost surely 

\begin{equation*}
\dim_H W_{\tau_1,\tau_2}=\begin{cases} \frac{1}{\tau_1} \ \ \ \ \ \ \ \ \ \ \ \ \ \ \ \  \ \ \ \ \ \ \ \ \ \ \ \ \ \ \ \ \ \ \ \ \ \ \ \text{ if } \ \frac{1}{\tau_1}\leq \dim_H \mu_0 -\dim_H \pi_2 \mu_0, \\ 
\frac{1}{\tau_1}-(\tau-1)(\beta_{\tau_1,\tau_2}-D_{\pi_2 \mu_0}(\beta_{\tau_1,\tau_2})) \text{ if } \ \dim_H \mu_0 -\dim_H \pi_2 \mu_0\leq \frac{1}{\tau_1}\leq v_{\tau}(\kappa_2), \\ \frac{1+(\tau_2 -\tau_1)(s_0 -2\kappa_2+D_{\pi_2 \mu_0}(\kappa_2)))}{\tau_2} \ \ \ \ \ \ \ \ \ \ \ \text{ if } \ \frac{1}{\tau_1}\geq v_{\tau}(\kappa_2),  \end{cases}
\end{equation*}
For a more precise statement (in particular regarding the value of $\kappa_2$), we refer to Theorem \ref{ThmRand} below.

An other very natural approximation problem of dynamical nature is the shrinking target problem. It was originally defined in \cite{HV} when the $``$ targets $"$ are balls.  Given a measurable mapping $T:\mathbb{R}^d \to \mathbb{R}^d$, $x\in\mathbb{R}^d$ and $\psi:\mathbb{N}\to \mathbb{R}_+$ it consists in studying $$\dim_H \left\{y\in\mathbb{R}^d : \ T^{n}(y)\in B(x,\psi(n))\text{ i.o. }\right\}:=E_{\psi}(x).$$ 
When $K$ is a two-dimensional base $b$-missing digit set, associated with the IFS $S=\left\{f_1,...,f_m\right\}$ and $T:K \to K$ is defined by $T(y)=by$, on can rewrite $$E_{\psi}(x)=\limsup_{n\geq 1, (i_1,...,i_n)\in\left\{1,...,m\right\}^n}B\Big(f_{i_1}\circ...\circ f_{i_n}(x), \psi(n)b^{-n}\Big)$$
and it was established by Beresnevitch and Velani that, for every $x\in K,$ one has $$\dim_H E_{\psi}(x)=\frac{\dim_H K}{1+\liminf_{n\to+\infty}\frac{\log \psi(n)}{-n\log b}}.$$
The shrinking target problem on fractals has known many developments since and an interested reader may refer to \cite{Baker,AllenB,ED4,RamsBara} for various related results. In the present article, we study the shrinking target when  $``$the targets $"$ are taken to be rectangles rather than balls. Unlike the case of balls, our result will depend, in general, on the choice of  $x\in K$, the center of our targets. More precisely, let $\nu$ be a $\times b$-ergodic measure supported on $K$, let $1\leq \tau_1 \leq \tau_2$ be two real numbers and write $$V_{\tau_1,\tau_2}(x)=\limsup_{n\geq 1, (i_1,...,i_n)\in\left\{1,...,m\right\}^n}\Big(f_{i_1}\circ...\circ f_{i_n}(x) +(-b^{-\tau_1 n},b^{-\tau_1  n})\times (-b^{-\tau_2  n},b^{-\tau_2  n})\Big).$$
Notice that $$V_{\tau_1,\tau_2}(x)=\left\{y\in K : \ T^{n}(y)\in\Big(x +(-b^{-(\tau_1 -1)n},b^{-(\tau_1 -1) n})\times (-b^{-(\tau_2 -1) n},b^{-(\tau_2 -1) n})\Big) \right\}.$$
We prove the following: for $\nu$-almost every $x$, one has $$ \dim_H V_{\tau_1,\tau_2}(x)= \min\left\{\frac{\dim_H K}{\tau_1}, \frac{\dim_H K +(\tau_2 -\tau_1)(\dim_H K -\alpha_{\nu})}{\tau_2}\right\},$$
where $\alpha_{\nu}$ is the almost sure local dimension of $\pi_2(x)$ with respect to $\pi_2 \mu_0$ (see Proposition \ref{PropoASLocdim}). We refer to Theorem \ref{ThmTree} for a more general statement and Corollary \ref{Coropsiteta}  for a formula holding for general approximation function along the $x$ and $y$-axis. Notice that set of possible dimensions (which are all attained) for $V_{\tau_1,\tau_2}$ when $\nu$ varies in the set of ergodic measures is $$\left\{\min\left\{\frac{\dim_H K}{\tau_1}, \frac{\dim_H K +(\tau_2 -\tau_1)(\dim_H K -\alpha)}{\tau_2}\right\}, \ \alpha\in \mbox{ Spectr}(\pi_2 \mu_0)\right\},$$
where $ \mbox{ Spectr}(\pi_2 \mu_0)$ denotes the set of possible local dimensions of the measure $\pi_2 \mu_0 .$ Finally, we mention that this result regarding the rectangular shrinking  targets problems was also established, independently, by Allen, Jordan and Ward in \cite{AJW} (see Remark \ref{RemarkRectanShr} for more details). 

\bigskip

In Section \ref{SecReca}, we recall the basis of geometric measure theory, theory of self-similar fractals and multifractal analysis. Our main results regarding the random covering by rectangles and the rectangular shrinking targets problem are stated in  Section \ref{Secmain} and the three last sections are dedicated to the proof of these theorems.

\section{Preliminaries and notations}

\label{SecReca}

Let us start with some notations 

 Let $d$ $\in\mathbb{N}$. For $x\in\mathbb{R}^{d}$, $r>0$,  $B(x,r)$ stands for the closed ball of ($\mathbb{R}^{d}$,$\parallel \ \ \parallel_{\infty}$) of center $x$ and radius $r$. 
 Given a ball $B$, $\vert B\vert$ stands for the diameter of $B$. For $t\geq 0$, $\delta\in\mathbb{R}$ and $B=B(x,r)$,   $t B$ stands for $B(x,t r)$, i.e. the ball with same center as $B$ and radius multiplied by $t$,   and the  $\delta$-contracted  ball $B^{\delta}$ is  defined by $B^{\delta}=B(x ,r^{\delta})$. 
\smallskip

Given a set $E\subset \mathbb{R}^d$, $\widering{E}$ stands for the  interior of the set $E$, $\overline{E}$ its  closure and $\partial E =\overline{E}\setminus \widering{E}$ its boundary. If $E$ is a Borel subset of $\R^d$, its Borel $\sigma$-algebra is denoted by $\mathcal B(E)$.
\smallskip

Given a topological space $X$, the Borel $\sigma$-algebra of $X$ is denoted $\mathcal{B}(X)$ and the space of probability measure on $\mathcal{B}(X)$ is denoted $\mathcal{M}(X).$ 

\sk

Given a metric space $X$ and $r>0.$ A $r$-packing of $X$ will consists of a set of open balls $\mathcal{T}$ such that for every $B\in\mathcal{T},$ $\vert B \vert=r$ and for every $L\neq B \in \mathcal{T},$ $L\cap B =\emptyset.$

\sk

 The $d$-dimensional Lebesgue measure on $(\mathbb R^d,\mathcal{B}(\mathbb{R}^d))$ is denoted by 
$\mathcal{L}^d$.
\smallskip

For $\mu \in\mathcal{M}(\R^d)$,   $\supp(\mu)=\left\{x\in \mathbb{R}^d: \ \forall r>0, \ \mu(B(x,r))>0\right\}$ is the topological support of $\mu$.
\smallskip

Given $X,Y $ two spaces endowed with $\sigma$-algebras and $\mu\in\mathcal{M}(X)$ a measure, for any measurable $f:X \to Y$, one will denote $f\mu \in\mathcal{M}(Y)$ the  measure $\mu \circ f^{-1}(\cdot).$ 

\smallskip
 Given $E\subset \mathbb{R}^d$, $\dim_{H}(E)$ and $\dim_{P}(E)$ denote respectively  the Hausdorff   and the packing dimension of $E$.
\smallskip

Given a set $S$, $\chi_S$ denotes the indicator function of $S$, i.e., $\chi_S(x)=1$ if $x\in S$ and $\chi_S(x)=0$ otherwise.

 \smallskip

 Given $b,n\in\mathbb{N},$ $\mathcal{D}_{b,n}$ or simply $\mathcal{D}_{n}$ when there is no ambiguity on $b$, denotes the set of $b$-adic cubes of generation $n$ and $\mathcal{D}$ the set of all $b$-adic cubes, i.e.
 \begin{equation*}
\mathcal{D}_{n}=\left\{b^{-n}(k_1,...,k_d)+b^{-n}[0,1)^d, \ (k_1,...,k_d)\in\mathbb{Z}^d\right\}\text{ and }\mathcal{D}=\bigcup_{n\geq 0}\mathcal{D}_{n}.
\end{equation*}
In addition, given $x\in\mathbb{R}^d,$ $D_{b,n}(x)$ or $D_n(x)$ will denote the $b$-adic cube of generation $n$ containing $x$.

\subsection{Recall on geometric measure theory}

\begin{definition}
\label{hausgau}
Let $\zeta :\mathbb{R}^{+}\mapsto\mathbb{R}^+$. Suppose that $\zeta$ is increasing in a neighborhood of $0$ and $\zeta (0)=0$. The  Hausdorff outer measure at scale $t\in(0,+\infty]$ associated with the gauge $\zeta$ of a set $E$ is defined by 
\begin{equation}
\label{gaug}
\mathcal{H}^{\zeta}_t (E)=\inf \left\{\sum_{n\in\mathbb{N}}\zeta (\vert B_n\vert) : \, \vert B_n \vert \leq t, \ B_n \text{ closed ball and } E\subset \bigcup_{n\in \mathbb{N}}B_n\right\}.
\end{equation}
The Hausdorff measure associated with $\zeta$ of a set $E$ is defined by 
\begin{equation}
\mathcal{H}^{\zeta} (E)=\lim_{t\to 0^+}\mathcal{H}^{\zeta}_t (E).
\end{equation}
\end{definition}

For $t\in (0,+\infty]$, $s\geq 0$ and $\zeta:x\mapsto x^s$, one simply uses the usual notation $\mathcal{H}^{\zeta}_t (E)=\mathcal{H}^{s}_t (E)$ and $\mathcal{H}^{\zeta} (E)=\mathcal{H}^{s} (E)$, and these measures are called $s$-dimensional Hausdorff outer measure at scale $t\in(0,+\infty]$ and  $s$-dimensional Hausdorff measure respectively. Thus, 
\begin{equation}
\label{hcont}
\mathcal{H}^{s}_{t}(E)=\inf \left\{\sum_{n\in\mathbb{N}}\vert B_n\vert^s : \, \vert B_n \vert \leq t, \ B_n \text{  closed ball and } E\subset \bigcup_{n\in \mathbb{N}}B_n\right\}. 
\end{equation}
The quantity $\mathcal{H}^{s}_{\infty}(E)$ (obtained for $t=+\infty$) is called the $s$-dimensional Hausdorff content of the set $E$.
\begin{definition} 
\label{dim}
Let $\mu\in\mathcal{M}(\mathbb{R}^d)$.  
For $x\in \supp(\mu)$, the lower and upper  local dimensions of $\mu$ at $x$ are  defined as
\begin{align*}
\underline\dim_{{\rm loc}}(\mu,x)=\liminf_{r\rightarrow 0^{+}}\frac{\log(\mu(B(x,r)))}{\log(r)}
 \mbox{ and } \ \    \overline\dim_{{\rm loc}}(\mu,x)=\limsup_{r\rightarrow 0^{+}}\frac{\log (\mu(B(x,r)))}{\log(r)}.
 \end{align*}
Then, the lower and upper Hausdorff dimensions of $\mu$  are defined by 
\begin{equation}
\label{dimmu}
\dimm(\mu)={\mathrm{ess\,inf}}_{\mu}(\underline\dim_{{\rm loc}}(\mu,x))  \ \ \mbox{ and } \ \ \overline{\dim}_P (\mu)={\mathrm{ess\,sup}}_{\mu}(\overline\dim_{{\rm loc}}(\mu,x))
\end{equation}
respectively.
\end{definition}

It is known (for more details see \cite{F}) that
\begin{equation*}
\begin{split}
\dimm(\mu)&=\inf\{\dim_{H}(E):\, E\in\mathcal{B}(\mathbb{R}^d),\, \mu(E)>0\} \\
\overline{\dim}_P (\mu)&=\inf\{\dim_P(E):\, E\in\mathcal{B}(\mathbb{R}^d),\, \mu(E)=1\}.
\end{split}
\end{equation*}
When $\underline \dim_H(\mu)=\overline \dim_P(\mu)$, this common value is simply denoted by $\dim(\mu)$ and~$\mu$ is said to be \textit{exact-dimensional}. 

Moreover, a measure $\mu \in\mathcal{M}(\mathbb{R}^d)$ is called Alfhors-regular if there exists $0\leq \alpha \leq d$ and $C>0$ such that for every $x\in\supp(\mu),$ for every $0<r \leq 1,$ one has $$C^{-1}r^{\alpha}\leq\mu\Big(B(x,r)\Big)\leq Cr^{\alpha}.$$
It is direct to check that such a measure is $\alpha$-exact dimensional.

\subsection{Self-similar measures and multifractal analysis}

Let us start by recalling the definition of a self-similar measure.

\begin{definition}
\label{def-ssmu}  A self-similar IFS is a family $S=\left\{f_i\right\}_{1 \leq i\leq m}$ of $m\geq 2$ contracting similarities  of $\mathbb{R}^d$. 

Let $(p_i)_{i=1,...,m}\in (0,1)^m$ be a positive probability vector, i.e. $p_1+\cdots +p_m=1$.

The self-similar measure $\mu$ associated with $ \left\{f_i\right\}_{1\leq i\leq m}$  and $(p_i)_{1\leq i \leq m}$ is the unique probability measure such that 
\begin{equation}
\label{def-ssmu2}
\mu=\sum_{i=1}^m p_i \mu \circ f_i^{-1}.
\end{equation}

The topological support of $\mu$ is the attractor of $S$, that is the unique non-empty compact set $K\subset X$ such that  $K=\bigcup_{i=1}^m f_i(K)$.

\end{definition}

The existence  and uniqueness of $K$ and $\mu$ are standard results \cite{Hutchinson}. Recall that due to a result by Feng and Hu \cite{FH}, any self-similar measure is exact dimensional.

\subsection{Multifractal analysis of self-similar measure satisfying OSC}

Let us start by defining the multifractal spectrum of a measure.

\begin{definition}
\label{defmultifractal}
Let $\mu \in\mathcal{M}(\mathbb{R}^d)$ be a measure and $h\geq 0.$ Set $$E_h=\left\{x\in\supp(\mu) \ : \ \lim_{r\to 0^+}\frac{\log \mu(B(x,r))}{\log r}=h. \right\}.$$
 The multifractal spectrum of $\mu$ is the mapping $D_{\mu}$, defined for every $h\geq 0$ by $$D_{\mu}(h)=\dim_H E_h.$$
Moreover, we call $$\mbox{Spectr}(\mu)=\overline{\left\{\alpha : \ D_{\mu}(\alpha)>0\right\}}.$$ 
  
\end{definition}

Let us also recall that a self-similar IFS $S=\left\{f_1,...,f_m\right\}$ is said to satisfy the open set condition if there exists a non empty open set $O$ such that 
\begin{equation}
\label{equaOSC}
\forall 1\leq i\neq j \leq m, \ f_{i}(O)\cap f_j (O)=\emptyset. 
\end{equation}
Given a self-similar IFS $S$ satisfying the open set condition, we will also say, by extension, that a self-similar measure $\mu$ associated with $S$  satisfies the open set condition.

The multifractal spectrum of any such measure is well understood. Fix $S\left\{f_1,..,f_m\right\}$ a self-similar IFS and $(p_1,...,p_m)\in(0,1)^m$ a probability vector. For $1\leq i \leq m,$ let $0<c_i <1$ be the contraction ratio of $f_i$ and, given $q\in\mathbb{R},$ set $$p_{i,q}=c_i^{-T(q)} p_i ^q,$$
where $T(q)$ is such that $$\sum_{1 \leq i \leq m}p_{i,q}=1.$$
Call also $\mu_q$ the self-similar measure associated with $(p_{i,q})_{1\leq i\leq m}$ and 

\begin{equation}
\begin{cases}\theta_{q}=\frac{\sum_{1\leq i\leq m}p_{i,q}\log p_{i,q}}{\sum_{1\leq i\leq m}p_{i,q}\log c_i} \\ \kappa_{q}=\frac{\sum_{1\leq i\leq m}p_{i,q}\log p_{i}}{\sum_{1\leq i\leq m}p_{i,q}\log c_i}.\end{cases}
\end{equation}

 We recall some of the properties of this spectrum.

\begin{proposition}[\cite{Fa1}, pages 286-295]
\label{propriOSC}
Let $\mu \in\mathcal{M}(\mathbb{R}^d)$ be a self-similar measure satisfying the open set condition. Then:
\begin{itemize}
\item[•] for any $\alpha\geq 0,$ $ D_{\mu}(\alpha)\leq \alpha,$ \medskip
\item[•] the mapping $\alpha \mapsto D_{\mu}(\alpha)$  is concave and reaches its maximum on $\alpha\geq \dim_H \mu$  such that $D_{\mu}(\alpha)=\dim_H K.$\medskip
\item[•] $\mbox{Spectr}(\mu)$ is a compact interval. Moreover $\mbox{Spectr}(\mu)=\left\{s\right\}\Leftrightarrow$ $\mu$ is Alfhors-regular,
\medskip
\item[•] if $\mu$ is not Alfhors-regular, then $D_{\mu}$ is $\mathcal{C}^{\infty}$ on $\overset{\circ}{\mbox{Spectr}(\mu)}=(\alpha_{\min},\alpha_{\max}).$ Moreover $D'_{\mu}$ is non increasing on $(\alpha_{\min},\alpha_{\max})$ and 
$$\begin{cases}\lim_{\alpha \to \alpha_{\min}}D'_{\mu}(\alpha)=+\infty\\ \lim_{\alpha \to \alpha_{\max}}D'_{\mu}(\alpha)=-\infty.\end{cases}$$ 
In addition, for any $q\in\mathbb{R},$ $D_{\mu}(\kappa_q)=\theta_q$, $D'_{\mu}(\kappa_q)=q$ and $\mu_q(E_{\kappa_q})=1.$
\end{itemize}
\end{proposition}
In addition of these properties, when $\mu$ is a self-similar measure satisfying the open set condition, it is also known that the multifractal spectrum and the so-called coarse multifractal spectrum coincides. More precisely, we have the following large deviation estimates.
\begin{theoreme}
\label{Thmlargedev}
Let $\mu\in\mathcal{M}(\mathbb{R}^d)$ be a self-similar measure satisfying the open set condition. Write $\mbox{Spectr}(\mu)=[\alpha_{\min},\alpha_{\max}].$ Let $\alpha_{\min}\leq \alpha\leq  \alpha_{\max} $ be a real number and, given $r>0$ let us write $$\mathcal{P}_{\alpha}(r,\varepsilon)=\sup\#\left\{\mathcal{T} \right\},$$
where $\mathcal{T}$ is a maximal $r$-packing of $\supp(\mu)$ with, for every $B \in\mathcal{T}$, $$\vert B \vert^{\alpha+\varepsilon} \leq \mu(B)\leq \vert B \vert^{\alpha-\varepsilon}.$$ 
There exists $r_{\alpha}>0$ such that for every $r\leq r_{\alpha},$ one has $$r^{-D_{\mu}(\alpha)+\varepsilon}\leq \mathcal{P}_{\alpha}(r,\varepsilon)\leq r^{-D_{\mu}(\alpha)-\varepsilon} .$$
\end{theoreme}

\section{Main results}
\label{Secmain}
\subsection{Random covering of self-similar carpet by rectangles}
Let us start by defining base-$b$ missing digit IFS's.

\begin{definition}
\label{Defmiss}
Let $b \in\mathbb{N}$ be an integer and  $\mathcal{A} \subset  \left\{0,...,b-1\right\}^2.$ Given $(i,j)\in\mathcal{A},$ let $g_{(i,j)}$ be the canonical contraction from $[0,1]^2$ to $(\frac{i}{b},\frac{j}{b})+[0,\frac{1}{b}]^2.$ The IFS $$S_{\mathcal{A}}:=\left\{g_{(i,j)}\right\}_{(i,j)\in\mathcal{A}}$$
is called a base-$b$ missing digit IFS. Its attractor $K$ is called a two-dimensional base $b$-missing digit set.
\end{definition}
Note that it is direct to check that $S_{\mathcal{A}}$ satisfies the open set condition, with $O=(0,1)^2.$

Given $S_{\mathcal{A}}$ a base-$b$ missing digit IFS, we define for every $0\leq i\leq b-1,$ 
$$p_{i, \mathcal{A}} =\frac{\#\left\{0\leq j\leq b-1 : \ (j,i)\in\mathcal{A}\right\}}{\# \mathcal{A}} \in [0,1].$$
Let $\mu_0$ be the Alfhors-regular self-similar measure on $K$, i.e. the self-similar  measure solution to $$\mu_0(\cdot)=\sum_{(i,j)\in\mathcal{A}}\frac{1}{\#\mathcal{A}}g_{(i,j)}\mu_0 (\cdot).$$

It is easily seen that the orthogonal projection of $\mu_0$ on the $y$-axis,  $\pi_2 \mu_0,$ is a self-similar measure associated with the IFS $\left\{g_{0,i}\right\}_{0 \leq i \leq b-1}$ and the weights $(p_{0,i}=p_{i, \mathcal{A}})_{0\leq i\leq b-1}.$

Given $q\in\mathbb{R}$ and $0\leq i\leq b-1,$ we also set $$p_{i,\mathcal{A},q}=\frac{p_{i,\mathcal{A}} ^q}{\sum_{0\leq j\leq b-1}p_{j,\mathcal{A}}^q}\text{ and } \kappa_{q,\mathcal{A}}=\frac{-\sum_{0\leq i\leq b-1}p_{i,\mathcal{A},q}\log p_{i,\mathcal{A}}}{\log b} .$$

Regarding the random and dynamical covering by rectangles of $K$, our main result is the following.

\begin{theoreme}
\label{ThmRand}
Let $\frac{1}{s_0}\leq \tau_1 \leq \tau_2$ be two real numbers and $S$ a base-$b$ two dimensional missing digit IFS. Let  $(X_n)_{n\in\mathbb{N}}$ be either an i.i.d. sequence of law $\mu_0$ or an orbit $(b^{n}x)_{n\in\mathbb{N}}, $ where $x\in\mathbb{T}^2,$ and 

$$W_{\tau_1,\tau_2}=\limsup_{n \to+\infty}\Big(X_n+(-\frac{1}{n^{\tau_1}}, \frac{1}{n^{\tau_1}})\times  (-\frac{1}{n^{\tau_2}}, \frac{1}{n^{\tau_2}})\Big).$$

 Write $\tau =\frac{\tau_2}{\tau_1}$ and define $v_{\tau}: \mbox{Spectr}(\pi_2(\mu_0)) \to \mathbb{R}$ by $$v_{\tau}(\alpha)=s_0 +(\tau-2)\alpha -(\tau-1)D_{\pi_2 \mu_0}(\alpha).$$ 
It is easily verified that $v_{\tau}$ is non increasing on $(-\infty, \kappa_{\frac{\tau -2}{\tau-1}}].$ Define, when possible (in particular when $\tau \neq 1$), $\beta_{\tau_1,\tau_2} $ as the unique solution on $[-\infty, \kappa_{\frac{\tau -2}{\tau-1}}]$ to  $v_{\tau}(\alpha)=\frac{1}{\tau_1}. $

Then, almost surely (or for $\mu_0$-almost every $x \in\mathbb{T}^2$):

\begin{equation}
\dim_H W_{\tau_1,\tau_2}=\begin{cases} \frac{1}{\tau_1} \ \ \ \ \ \ \ \ \ \ \ \ \ \ \ \  \ \ \ \ \ \ \ \ \ \ \ \ \ \ \ \ \ \ \ \ \ \ \ \text{ if } \ \frac{1}{\tau_1}\leq \dim_H \mu_0 -\dim_H \pi_2 \mu_0, \\ 
\frac{1}{\tau_1}-(\tau-1)(\beta_{\tau_1,\tau_2}-D_{\pi_2 \mu_0}(\beta_{\tau_1,\tau_2})) \text{ if } \ \dim_H \mu_0 -\dim_H \pi_2 \mu_0\leq \frac{1}{\tau_1}\leq v_{\tau}(\kappa_2), \\ \frac{1+(\tau_2 -\tau_1)(s_0 -2\kappa_2+D_{\pi_2 \mu_0}(\kappa_2)))}{\tau_2} \ \ \ \ \ \ \ \ \ \ \ \text{ if } \ \frac{1}{\tau_1}\geq v_{\tau}(\kappa_2),  \end{cases}
\end{equation}
\end{theoreme}

\begin{remark}
\begin{itemize}
\item[(1)] For simplicity, the results where formulate in the case  $\tau_1 \leq \tau_2$, but the results straightforwardly adapts to the case $\tau_2 \leq \tau_1$ by switching the roles of $\tau_1$ and $\tau_2$ and considering $\pi_1 \mu_0$ rather than $\pi_2 \mu_0.$ \medskip
\item[(2)] $\frac{1}{s_0}\leq \tau_1 =\tau_2,$ one recovers that $\dim_H W_{\tau_1,\tau_1}=\frac{1}{\tau_1}.$\medskip
\item[(3)] When $S$ has uniform fibers, meaning that $p_{i,\mathcal{A}}=p_{j,\mathcal{A}}$ for every $0 \leq i,j\leq b-1,$ then $\pi_2 \mu_0$ is Alfhors-regular, which implies in particular that $ \mbox{Spectr}(\pi_2 \mu_0)=\left\{\dim_H \pi_2 \mu_0\right\}$ and $\dim_H \pi_2 \mu_0 =D_{\pi_2 \mu_0}(\dim_H \pi_2 \mu_0).$ Thus in this case, one obtains for every $\frac{1}{s_0}\leq \tau_1 \leq \tau_2,$ $$\dim_H W_{\tau_1,\tau_2}=\min\left\{\frac{1}{\tau_1},\frac{1+(\tau_2-\tau_1)(s_0 -\dim_H \pi_2 \mu_0)}{\tau_2}\right\},$$
which is consistent with the sub-case where $\mathcal{A}=\mathcal{A}_1 \times \mathcal{A}_2,$ where $\mathcal{A}_1, \mathcal{A}_2 \subset \left\{0,...,b-1\right\}.$
\end{itemize}

\end{remark}

The next section presents our result regarding the rectangular shrinking target problem.

\subsection{Tree approximation}

In this section, we study the $``$ tree   approximation$"$, which, as mentioned in introduction, can be seen as a reformulation of the classical shrinking target problem associated with the the mapping $T_b :\mathbb{T}^2 \to \mathbb{T}^2,$ defined by $T_b(x)=bx.$

Consider again $S=\left\{f_1,..,f_m\right\}$, a two dimensional  base-$b$ missing digit IFS of attractor $K$  and $\mu_0$ the Alfhors-regular self-similar measure on $K$. The following proposition is necessary to state our main result and  will be established in the next section, as Proposition \ref{PropLocDimCenter} applied with $\pi_2 \nu$ (which can be identified with a $\times b$ ergodic measure on $\mathbb{T}^1$).

\begin{proposition}
\label{PropoASLocdim}
Let $\nu$ be a $\times$-b (i.e. with respect to $T_b$) ergodic measure. Then, there exists $\alpha_{\nu}$ such that, for $\nu$-almost every $x$, one has $$\lim_{r\to 0^+}\frac{\log \pi_2 \mu_0 B\Big(\pi_2(x),r\Big)}{\log r}=\alpha_{\nu}.$$ 
\end{proposition}

In the next theorem, given a word $\underline{i}=(i_1,...,i_n)\in\left\{1,...,m\right\}^n,$ one writes $$f_{\underline{i}}=f_{i_1}\circ ... \circ f_{i_n}. $$  
\begin{theoreme}
\label{ThmTree}
Let $\mu$ be a self-similar measure (with respect to $S$) and $\Lambda \subset \bigcup_{k\geq 1}\left\{1,...,m\right\}^k$ be a set of words such that $$\mu\Big(\limsup_{\underline{i}\in\Lambda}f_{\underline{i}}([0,1]^2)\Big)=1.$$
Then, for every $1\leq \tau_1 \leq \tau_2,$ for any $\times b$ ergodic measure $\nu$ with $\supp(\nu)\subset K$, writing 
$$V_{\tau_1,\tau_2}(x)=\limsup_{\underline{i}\in\Lambda}\Big( f_{\underline{i}}(x)+(-\vert f_{\underline{i}}(K)\vert^{\tau_1},\vert f_{\underline{i}}(K)\vert^{\tau_1})\times (-\vert f_{\underline{i}}(K)\vert^{\tau_2},\vert f_{\underline{i}}(K)\vert^{\tau_2})\Big),$$
  for $\nu$-almost every $x$, one has

$$\dim_H V_{\tau_1,\tau_2}(x)\geq \min\left\{\frac{\dim_H \mu}{\tau_1}, \frac{\dim_H \mu +(\tau_2 -\tau_1)(s_0 -\alpha_{\nu})}{\tau_2}\right\}.$$

Assume in addition that $$\lim_{\vert\underline{i}\vert \to +\infty}\frac{ -\log \mu(f_{\underline{i}}([0,1]^2))}{\vert \underline{i}\vert \log b}=\dim_H \mu,$$
then , for $\nu$-almost every $x$,
\begin{align*}
\dim_H V_{\tau_1,\tau_2}(x)= \min\left\{\frac{\dim_H \mu}{\tau_1}, \frac{\dim_H \mu +(\tau_2 -\tau_1)(s_0 -\alpha_{\nu})}{\tau_2}\right\}.
\end{align*}

\end{theoreme}

Given $\nu$ a $\times b$-ergodic measure $\nu$ with $\supp(\nu)\subset K$, $\psi,\theta:\mathbb{N} \to \mathbb{R}_+,$ define $$\begin{cases}\mathcal{N}_1 =\left\{n : \ \psi(n) \geq \theta(n)\right\} \\  \mathcal{N}_2 =\left\{n : \ \psi(n) < \theta(n)\right\}\end{cases} $$
and $\beta_{\nu}$ the real number (which exists by Proposition \ref{PropLocDimCenter}) such that, for $\nu$-almost every $x$, one has

$$\lim_{r\to 0^+}\frac{\log \pi_1 \mu_0 B\Big(\pi_1(x),r\Big)}{\log r}=\beta_{\nu}.$$

Define also

\begin{align*}
&\limsup_{n\in\mathcal{N}_1} \min\left\{\frac{s_0}{\frac{\log \psi(n)}{-n \log b}}, \frac{s_0 +\Big(\frac{\log \theta(n)}{-n \log b} -\frac{\log \psi(n)}{-n \log b}\Big)(s_0 -\alpha_{\nu})}{\frac{\log \theta(n)}{-n \log b}}\right\}\\
&:=g_1(\psi,\theta,\nu)\\
&\limsup_{n\in\mathcal{N}_2} \min\left\{\frac{s_0}{\frac{\log \theta(n)}{-n \log b}}, \frac{s_0 +\Big(\frac{\log \psi(n)}{-n \log b} -\frac{\log \theta(n)}{-n \log b}\Big)(s_0 -\beta_{\nu})}{\frac{\log \psi(n)}{-n \log b}}\right\}\\
&:=g_2(\psi,\theta,\nu)
\end{align*}

and, for $i=1,2,$ consider two non increasing sequences of integers $ (n_{k,i})_{k\in\mathbb{N}}\subset \mathcal{N}_i ^{\mathbb{N}}$ such that
\begin{align*}
&\lim_{k\to +\infty }\min\left\{\frac{s_0}{\frac{\log \psi(n_{k,1}) }{-n_{k,1} \log b}}, \frac{s_0 +\Big(\frac{\log \theta(n_{k,1} )}{-n_{k,1}  \log b} -\frac{\log \psi(n_{k,1} )}{-n_{k,1}  \log b}\Big)(s_0 -\alpha_{\nu})}{\frac{\log \theta(n_{k,1} )}{-n_{k,1}  \log b}}\right\}\\
&:=g_1(\psi,\theta,\nu)\\
&\text{ and }\\
& \lim_{k\to +\infty }\min\left\{\frac{s_0}{\frac{\log \theta(n_{k,2}) }{-n_{k,2} \log b}}, \frac{s_0 +\Big(\frac{\log \psi(n_{k,2} )}{-n_{k,2}  \log b} -\frac{\log \theta(n_{k,2} )}{-n_{k,2}  \log b}\Big)(s_0 -\beta_{\nu})}{\frac{\log \psi(n_{k,2} )}{-n_{k,2}  \log b}}\right\}\\
&:=g_2(\psi,\theta,\nu).
\end{align*}
By applying \ref{ThmTree} successively to $\Lambda_1=\bigcup_{k\geq 1}\left\{1,...,m\right\}^{n_{k,1}}$ and $\Lambda_1=\bigcup_{k\geq 1}\left\{1,...,m\right\}^{n_{k,2}}$ and $\mu_0$ , one obtains the following corollary.  
\begin{corollary}
\label{Coropsiteta}
Let $\psi,\theta:\mathbb{N}\to \mathbb{R}_+$ be two mappings such that $$\min\left\{\liminf_{n\to +\infty}\frac{\log \psi(n)}{-n\log b},\liminf_{n\to +\infty}\frac{\log \theta(n)}{-n\log b}\right\}\geq 1.$$ Then 

\begin{align*}
&\dim_H \limsup_{\underline{i}\in\bigcup_{k\geq 1}\left\{1,...,m\right\}^k}\Big( f_{\underline{i}}(x)+(-\psi(\vert \underline{i}\vert),\psi(\vert \underline{i}\vert)\times (-\theta(\vert \underline{i}\vert),\theta(\vert \underline{i}\vert))\Big)\\
&=\max\left\{g_1(\psi,\theta,\nu),g_2(\psi,\theta,\nu)\right\}.
\end{align*}

\end{corollary}

\begin{remark}
\label{RemarkRectanShr}
\item[•] Interestingly, unlike in the case of balls (i.e. for $\tau_1=\tau_2$), Theorem \ref{ThmTree} shows that $\dim_H V_{\tau_1,\tau_2}(x)$ depends in general on the choice of $x\in K.$ In particular, under these settings, the set of possible values (which are all attained) is $$\left\{\min\left\{\frac{\dim_H \mu}{\tau_1}, \frac{\dim_H \mu +(\tau_2 -\tau_1)(s_0 -\alpha)}{\tau_2}\right\} \  \alpha\in \mbox{Spectr}(\pi_2 \mu_0)\right\}.$$

\item[•] A careful reader will notice that  the proof of Theorem \ref{ThmTree} only requires $\pi_2 \nu$ to be ergodic rather than $\nu.$ Thus the assumption of Theorem \ref{ThmTree} can be weakened accordingly.\medskip 
\item[•] As mentioned in the introduction, Corollary \ref{Coropsiteta} was also obtained by Allen, Jordan and Ward in \cite{AJW}, using a different method, under the  weaker assumption that the coding of $x=(x_n :=(x_n^{1},x_n^{2}))_{n\in\mathbb{N}}\in \mathcal{A}^{\mathbb{N}}$ satisfies that, writing $S=S_{\mathcal{A}},$ for each $0\leq j\leq b-1$ for which there exists $0 \leq i \leq b-1$ so that $(i,j)\in\mathcal{A},$ one has $$\lim_{n \to +\infty}\frac{\#\left\{0\leq k\leq n : x_n ^2 =j \ \right\}}{n}=\kappa_j,$$
for some $0\leq \kappa_j < 1 .$
\end{remark}

As a second application, we study the anisotropic approximation under digit frequency constraints. Such problems where for instance studied in \cite{BS2,ED1} To this end, we first recall a corollary of \cite[Proposition 2.2]{BS2}

\begin{proposition}[\cite{BS2}]
Let $(p_1,...,p_m)$ be a probability vector and
\begin{align*}
&\Lambda_{(p_1,...,p_m)}\\
&=\left\{\underline{i}=(i_1,...,i_n) : \ \forall 1 \leq k\leq m, \Big\vert \frac{1}{n}\#\left\{1\leq j\leq m : i_j =k\right\}-p_k \Big\vert \leq \sqrt{\frac{2\log \log n}{n}}\right\}.
\end{align*}
 
Then, if $\mu$ is the self-similar measure associated with $(p_1,..,p_m),$ one has $$\mu\Big(\limsup_{\underline{i}\in\Lambda_{(p_1,...,p_m)}}f_{\underline{i}}\Big([0,1]^2\Big)\Big)=1.$$
\end{proposition}

Our result related to approximation under digit frequencies is the following.

\begin{corollary}
Let $(p_1,...,p_m)$ be a probability vector and
\begin{align*}
&\Lambda_{(p_1,...,p_m)}\\
&=\left\{\underline{i}=(i_1,...,i_n) : \ \forall 1 \leq k\leq m, \Big\vert \frac{1}{n}\#\left\{1\leq j\leq m : i_j =k\right\}-p_k \Big\vert \leq \sqrt{\frac{2\log \log n}{n}}\right\}.
\end{align*}
Then 
\begin{align*}
&\dim_H \limsup_{\underline{i}\in\Lambda_{(p_1,...,p_m)}}\Big( f_{\underline{i}}(x)+(-\vert f_{\underline{i}}(K)\vert^{\tau_1},\vert f_{\underline{i}}(K)\vert^{\tau_1})\times (-\vert f_{\underline{i}}(K)\vert^{\tau_2},\vert f_{\underline{i}}(K)\vert^{\tau_2})\Big)\\
&= \min\left\{\frac{\frac{-\sum_{1\leq i\leq m}p_i \log p_i}{\log b}}{\tau_1}, \frac{\frac{-\sum_{1\leq i\leq m}p_i \log p_i}{\log b} +(\tau_2 -\tau_1)(s_0 -\alpha_{\nu})}{\tau_2}\right\}.
\end{align*}
\end{corollary}

The next three sections are dedicated to the proofs of Theorem \ref{ThmRand} and Theorem \ref{ThmTree}.  The Section \ref{SecP} establishes some useful general results regarding $\times b$ ergodic measures, provides some recalls on the mass transference principle for self-similar measures and establishes important content estimates, useful to the proof of both theorems. In the last section, we finally prove Theorem \ref{ThmRand} and Theorem \ref{ThmTree}.

%
%
%
%
%
%

\section{Some preliminaries to the proof of Theorems \ref{ThmRand} and Theorem \ref{ThmTree}}
\label{SecP}
\subsection{Results regarding $\times b$ ergodic measures}
\label{SectionPreli}

Let us fix $b\in\mathbb{N}$ and $\eta $  a $\times b$ ergodic measure on $\mathbb{T}$. For $0\leq i\leq b-1,$ let us denote $f_i$ the mapping defined for every $x\in\mathbb{T}$ by $f_i(x)=\frac{x+i}{b}$ and $S_b= \left\{f_i\right\}_{0\leq i\leq b-1}.$ The goal of this section is to prove the following result.

\begin{proposition}
\label{PropLocDimCenter}
Let $\mu $ be a self-similar measure associated with the IFS $S_b$ and the probability vector $(p_0,...,p_{b-1})$. Then for $\eta$-almost every $x$, one has $$\lim_{n\to +\infty}\frac{\log \mu\Big( D_{b,n}(x)\Big)}{-n\log b}=\lim_{r\to 0^+}\frac{\log \mu\Big(B(x,r)\Big)}{\log r}=\frac{-\sum_{0\leq i\leq b-1}\eta([\frac{i}{b},\frac{i+1}{b}])\log p_i}{\log b}.$$
\end{proposition}
It is worth mentioning that equality between the left-handside and the right-handside can be obtained as a consequence of Birkhoff's ergodic theorem. The difficulty here comes from the case where one considers a centered ball, as $\mu$ is not assume to be doubling. The proof of Proposition \ref{PropLocDimCenter} relies on the following lemma.

\begin{lemme}
\label{PropoApproxErgodb}
Fix $\tau>1 $ and write $$A_{b,\tau}= \left\{x: \ \vert x-\frac{k}{b^n}\vert \leq \frac{1}{b^{n\tau}}\text{ i.o. }\right\}.$$

If $\eta$ has no atom then $$\eta\Big(A_{b,\tau}\Big)=0.$$
\end{lemme}

Before proving Lemma \ref{PropoApproxErgodb}, we show how it implies Proposition \ref{PropLocDimCenter}. Prior to that we start by a small classical lemma.

\begin{lemme}
\label{lemmadiffu}
Let $T:\mathbb{R}^d \to \mathbb{R}^d$ be a measurable mapping and $\mu\in\mathcal{M}(\mathbb{R}^d)$ a $T$-ergodic probability measure. Then $\mu$ is  diffuse (i.e. has no atom) or carried by a periodic orbit (in which case it is purely atomic with finitely many atoms).
\end{lemme}
\begin{proof}
Assume that there exists a measurable set $A$ with $\mu(A)>0$ and for any $x\in A,$ $\mu(\left\{x\right\})=0.$ Assume that there exists $y$ such that $\mu\Big(\left\{y\right\}\Big)>0.$ Then, for every $n\in\mathbb{N},$ as $$ \left\{y\right\}\subset T^{-n}\Big(\left\{T^{n}(y)\right\}\Big),$$
one has $\mu\Big(\left\{T^{n}(y)\right\}\Big)>0.$ In addition, by Poincarré's recurrence Theorem, there exists $n$ such that $T^{n}(y)\in A,$ so that $$ 0=\mu\Big(\left\{T^{n}(y)\right\}\Big)>0.$$
We conclude that there is no $y$ with $\mu\Big(\left\{y\right\}\Big)>0$ and $\mu$ is either purely atomic or diffuse. In the case where $\mu$ is purely atomic, from the fact that $\mu$ is an ergodic probability measure we deduce that $\mu$ must be carried by an eventually periodic orbit and from Poincarré's recurrence Theorem that this eventually periodic orbit must be a periodic orbit. 
\end{proof}

\begin{proof}

First, recall that  Since Proposition \ref{PropLocDimCenter} is direct if $\eta$ is  atomic, we assume now that $\eta$ has no atoms.
For every $x\in\mathbb{T}$ and $0\leq i\leq b-1,$ write 
$$S_{N,i}(x)=\frac{1}{N}\sum_{n=0}^{N-1} \chi_{[\frac{i}{b},\frac{i+1}{b}[}(b^n x) .$$
By Birkhoff ergodic's theorem, for $0\leq i\leq b-1,$ writing $\widetilde{p}_i =\eta([\frac{i}{b},\frac{i+1}{b}])$, for $\eta$-almost every $x$, for every $0 \leq i\leq b-1,$ $$S_{N,i}\to\widetilde{ p}_i.$$
Since $\mu$ is a Bernoulli measure, this implies that, for $\eta$-almost every $x$, $$\lim_{n\to +\infty}\frac{\log \mu\Big( D_{b,n}(x)\Big)}{-n\log b}=\frac{-\sum_{0\leq i\leq b-1}\widetilde{p_i}\log p_i}{\log b}:=\alpha.$$
Let us fix $\tau>1.$ By Proposition \ref{PropoApproxErgodb}, for $\eta$-almost every $x,$ there exists $N_x$ large enough so that for every $n\geq N_x,$ for every $k\in\mathbb{N},$ one has $$\vert x-\frac{k}{b^n}\vert \geq \frac{1}{b^{n\tau}} .$$
Thus we conclude that for every $n\geq N_x,$ one has $$B(x,b^{-n\tau})\subset D_{b,n}(x)\subset B(x,b^{-n}).$$
This yields $$\limsup_{r\to 0^+}\frac{\log \mu\Big(B(x,r)\Big)}{\log r}\leq\lim_{n\to +\infty}\frac{\log \mu\Big( D_{b,n}(x)\Big)}{-n\log b} \leq \tau \liminf_{r\to 0^+}\frac{\log \mu\Big(B(x,r)\Big)}{\log r}. $$
Since $\tau>1$ was arbitrary, letting $\tau\to 1$ along a countable sequence proves the claim.
\end{proof}

%
%

We now prove Lemma \ref{PropoApproxErgodb}.

\begin{proof}
Fix $\tau>1$ and assume that $$\eta\Big(A_{b,\tau}\Big)>0.$$ For every $0\leq i\leq b-1,$ for every $x\in\mathbb{T},$ $$S_{N,i}(x)=\frac{1}{N}\sum_{n=0}^{N-1} \chi_{[\frac{i}{b},\frac{i+1}{b})}(b^n x) $$
is an ergodic average. Thus writing $p_i =\eta\Big([\frac{i}{b},\frac{i+1}{b})\Big)$, for $\eta$-almost every $x\in A_{b,\tau},$ one has $$\lim_{N \to+\infty}S_{N,i}=p_i.$$ 
On the other hand, for every $x\in A_{b,\tau}\setminus \left\{\frac{k}{b^n}, n\in\mathbb{N} ,0\leq k\leq b^n\right\}$, there exists a sequence $(n_k)_{m\in\mathbb{N}}$ of integers such that for every $k\in\mathbb{N},$ $$\vert x-\frac{m_k}{b^{n_k}}\vert \leq \frac{1}{b^{\tau n_k}} $$
for some $m_k \in\mathbb{N}.$ This, depending on the sign of $x-\frac{m_k}{b^{n_k}}$ implies that the coding on base $b$ of $x$ contains at least $(\tau -1)n_k$ $0$ or $b-1$ in $n_k$-th position. Thus, for $j=0$ or $b-1,$  there exist a sub-sequence of integers $\widetilde{n}_k$ such that for every $k\in \mathbb{N},$ the expansion in base $b$ of $x$ contains $(\tau-1)\widetilde{n}_k$ $j$ in $\widetilde{n}_k$-th position. Assume now that there exists $1\leq i\neq j\leq b-1$ such that $p_i>0$ and fix $\varepsilon>0$ and  $N$ large enough so that for every $n\geq N$, $$\vert S_{n,i} -p_i \vert \leq \varepsilon.$$
Fix $\widetilde{n}_k>N.$ we have  $$\vert S_{\widetilde{n}_k,i} -p_i \vert\leq \varepsilon$$
and the coding of $x$ contains $(\tau-1)\widetilde{n}_k$ $j \neq i$ in $\widetilde{n}_k$'th position. This yields that $$S_{\tau \widetilde{n}_k} \leq p_i +\varepsilon-\frac{\tau-1}{\tau}<p_i -\varepsilon$$ provided that $\varepsilon$ was chosen small enough to begin with. We conclude that $p_j=1$ and $p_i=0$ for every $i\neq j.$
 
Recall that $x\in\mathbb{T}$ is called generic for $\eta$ if $\frac{1}{N}\sum_{k=0}^{N-1}\delta_{b^k x}$ 
 converges weakly to $\eta$ and that $\eta$-almost every $x$ is generic.  Thus for every interval $I\subset \mathbb{T}$, since $\eta(\partial I)=0,$ for any $\eta$-generic $x$, one has  
\begin{equation}
\label{equaibar}
\lim_{N \to+\infty}\frac{1}{N}\sum_{0\leq k\leq N-1}\chi_{I}(b^n x)=\eta(I).
\end{equation}
  
 Now, fix $x$ $\eta$-generic and such that $S_{N,i}(x)\to p_i$ for every $0\leq i\leq b-1$ and $\underline{i}=(i_1,...,i_m)\in\left\{0,...,b-1\right\}^m $. Assume that there exists $1\leq k\leq m$ such that $i_k \neq j$, call $I_{\underline{i}}$ the projection in base $b$ of the cylinder $[\underline{i}]$ on $\mathbb{T}.$ If $\eta(I_{\underline{i}})>0,$ by \eqref{equaibar}, writing $(x_n)_{n\in\mathbb{N}}$ the coding of $x$ in base $b$, one gets
$$ \lim_{N \to+\infty}\frac{1}{N}\sum_{0\leq n\leq N-1}\chi_{I_{\underline{i}}}(b^n x)= \lim_{N \to+\infty}\frac{1}{N}\sum_{0\leq n\leq N-1}\chi_{\underline{i}}((x_n,...,x_{n+m-1}))=\eta(I_{\underline{i}}).$$ 
This implies in particular that 
  
  $$\liminf_{N \to+\infty}\frac{1}{N}\sum_{0\leq n\leq N-1}\chi_{[\frac{i_k}{b},\frac{i_k +1}{b}[}(b^n x)=\liminf_{N \to+\infty}\frac{1}{N}\sum_{0\leq n\leq N-1}\chi_{i_k}(x_n)\geq \frac{\eta(I_{\underline{i}})}{2m}>0,$$
which is a contradiction. Thus   $\eta([\underline{i}])=0$ whenever there exists $i_k$ such that $i_k \neq j$ but this implies $\eta= \delta_0$. The assumption  $\eta\Big(A_{b,\tau}\Big)>0$ was therefore false.   
\end{proof}
We also isolate the following corollary, which will be useful later on.
\begin{corollary}
\label{Corodya}
Let $\nu \in\mathcal{M}(\mathbb{T}^1)$ be a non atomic $\times$ b ergodic measure. Then for every $0<\varepsilon<1,$ for $\eta$-almost every $x\in\mathbb{T}^1,$ there exists $n_{x,\varepsilon}>0$ such that for every $n\geq n_{x,\varepsilon},$ one has  $$ D_{b,n}(x)\subset B(x,b^{-n})\subset D_{b, \lfloor (1-\varepsilon)n \rfloor}(x) $$
\end{corollary}


\subsection{Recalls on mass transference principle for self-similar measures}
This section is dedicated to some recall on a key result one will use in order to prove Theorem \ref{ThmRand} and Theorem \ref{ThmTree}. 

\begin{theoreme}[\cite{ED3}]
\label{MTPss}
Let $\mu \in\mathcal{M}(\mathbb{R}^d)$ be a self-similar measure. Let $(B_n)_{n\in\mathbb{N}}$ be a sequence of balls such that $\vert B_n \vert \to 0$  and $$\mu\Big(\limsup_{n\to+\infty}B_n\Big)=1.$$
Let $0\leq s<\dim_H \mu$ and $(U_n)_{n\in\mathbb{N}}$ be a sequence of open sets such that for every $n\in\mathbb{N}$,
\begin{itemize}
\item[•] $U_n \subset B_n,$ \medskip
\item[•] $\mathcal{H}^{s}_{\infty}\Big(U_n \cap \supp(\mu)\Big)\geq \mu\Big(B_n\Big).$
\end{itemize}
Then $\dim_H \limsup_{n\to+\infty}U_n \geq s.$
\end{theoreme}

\subsection{Content estimates}
\label{SecContEsti}
Let us fix $S$ a $b$-adic IFS, call $\mu_0$ the corresponding Alfhors-regular self-similar measure and $\widetilde{K}$ its attractor.  In this section, in order to avoid possible confusion between $b$-adic intervals and $b$-adic cubes of $[0,1]^2,$ we will denote $I$ $b$-adic intervals and $D$ $b$-adic cubes. In addition, we will identify when necessary the measure $\pi_2 \mu_0$ with its one dimensional natural counterpart. 

Let us start by remarking that, if there exists a $b$-adic cube $D$ such that $\mu_0(\partial D)>0,$ then either $\pi_1 \mu_0$ or $\pi_2 \mu_0$ has an atom. Since these measures are self-similar, one obtains that this measure is a Dirac mass so that $\mu_0$ is supported on an horizontal or vertical line. \textbf{In the rest of the section, we assume that $\mu_0$ is not supported on an horizontal or vertical line, hence satisfies that $\mu_0(\partial D)=0$ for any $b$-adic  cube $D$}.  

Given $\eta \in\mathcal{M}(\mathbb{R}^d)$ a Borel set $A \in\mathcal{M}(\mathbb{R}^d)$ and $s\geq 0,$ define $$\mathcal{H}^{\eta,s}_{\infty}(A)=\inf\left\{\mathcal{H}^s_{\infty}(E), \ E\subset A : \ \mu(E)=\mu(A)\right\}.$$

The following result, established as \cite[Theorem 2.6]{ED3}.

\begin{theoreme}
\label{contss}
Let $\Omega$ be an open set and $0\leq s \leq\dim_H \mu_0:=s_0$. Then there exists a constant $C$ (depending only on $\mu_0$ and $s$) such that 
\begin{equation}
\label{EquaMucontss}
C\mathcal{H}^{s}_{\infty}\Big(\Omega \cap \widetilde{ K}\Big)\leq \mathcal{H}^{\mu_0,s}_{\infty}\Big(\Omega\Big)\leq \mathcal{H}^{s}_{\infty}\Big(\Omega \cap \widetilde{ K}\Big).
\end{equation}
\end{theoreme}

Write $$F=\bigcup_{n \geq 0 , D\in\mathcal{D}_n}\partial D.$$
As a consequence of Theorem \ref{contss} and the fact that $\mu_0(F)=0,$ there  exists a constant $C>0$ such that for any open sets $\Omega$  $$\mathcal{H}^{s}_{\infty}(\Omega \cap K)\leq\mathcal{H}^{s}_{\infty}(\Omega \cap \widetilde{K}) \leq C \mathcal{H}^{s}_{\infty}(\Omega \cap K),$$
where $$K=\widetilde{K}\setminus F.$$ 
Moreover, since $S$ satisfies the open set condition with open set $(0,1)^2$, it is direct to verify that for any open set $\Omega \subset [0,1]^2$ and any $b$-adic cube $D$, one has $$f_{D}(\Omega \cap K)=f_D(\Omega)\cap K,$$
where $f_D$ denotes the canonical contraction from $[0,1)^2$ to $D$.

\begin{proposition}
\label{MainPropoCont}
Let $I $ be a $b$-adic interval, $x\in [0,1]\times I\cap K$ and let us write $$\mbox{Str}(I)=[0,1]\times I.$$
Let $K$ be a base $b$ missing digit set. Then for every $0<s\leq \dim_H K$, one has $$\frac{C_{s}}{-\log \vert I \vert}\leq \frac{ \mathcal{H}^{s}_{\infty}\Big(\mbox{Str}(I)\cap K\Big)}{\min_{0\leq \gamma \leq 1}b^{-(s-s_0)\lfloor \gamma n\rfloor}\times\pi_2\mu_0(D_{\lfloor \gamma n\rfloor}(\pi_2(x)))}\leq 1$$
\end{proposition}

\begin{proof}
Notice that, because $\pi_2$ induces a semi-conjugacy between $S$ and the projected IFS on $\left\{0\right\}\times [0,1]$ (with the base $b$ shifts), denoting $\sigma$ both of these shifts, one has $$\begin{cases}\mu_0(\mbox{Str}(I))=\pi_2 \mu_0 (D_n(\pi_2(x))) \\ \mu_0 \Big(\mbox{Str}\Big(I_{\lfloor (1-\gamma)n\rfloor}(\pi_2(\sigma^{\lfloor \gamma n\rfloor}(x)))\Big)\Big)=\pi_2 \mu_0\Big(D_{\lfloor (1-\gamma)n\rfloor}(\sigma^{\lfloor \gamma n\rfloor}(\pi_2(x)))\Big)\end{cases}.$$
In addition, denoting $(p_i)_{1\leq i \leq b}$ the probability vector associated with the self-similar measure $\pi_2 \mu_0$ and  $(y_n)_{n\in\mathbb{N}}$ the coding in base $b$ of $\pi_2 (x),$ one has $$\begin{cases}\pi_2 \mu_0 (D_n(\pi_2(x))) =\prod_{k=1}^n p_{y_i} \\ \pi_2 \mu_0\Big(D_{\lfloor (1-\gamma)n\rfloor}(\sigma^{\lfloor \gamma n\rfloor}(\pi_2(x)))\Big)=\prod_{k=\lfloor \gamma n\rfloor+1}^{\lfloor \gamma n\rfloor +\lfloor (1-\gamma)n\rfloor}p_{y_i}. \end{cases} $$  
Thus there exists a constant $C>0$ such that for any $0\leq \gamma \leq 1,$ $$ C^{-1} \prod_{i=1}^{\lfloor \gamma n\rfloor}p_{y_i} \leq \frac{\mu_0\Big(\mbox{Str}(I)\Big)}{\mu_0 \Big(\mbox{Str}\Big(I_{\lfloor (1-\gamma)n\rfloor}(\pi_2(\sigma^{\lfloor \gamma n\rfloor}(x)))\Big)\Big)} \leq C \prod_{i=1}^{\lfloor \gamma n\rfloor}p_{y_i}=C \pi_2 \mu_0(D_{\lfloor \gamma n\rfloor}(\pi_2(x))).$$

Thus it is sufficient to show that 

$$\frac{C_{s}}{-\log \vert I \vert}\leq \frac{ \mathcal{H}^{s}_{\infty}\Big(\mbox{Str}(I)\cap K\Big)}{\min_{0\leq \gamma \leq 1}b^{-(s-s_0)\lfloor \gamma n\rfloor}\times \frac{\mu_0\Big(\mbox{Str}(I)\Big)}{\mu_0 \Big(\mbox{Str}\Big(I_{\lfloor (1-\gamma)n\rfloor}(\pi_2(\sigma^{\lfloor \gamma n\rfloor}(x)))\Big)\Big)}}\leq 1.$$

 We first establish the lower-bound. Let us first notice that $$\mbox{Str}(I)=\bigcup_{D\in\mathcal{D}_{b,n}: D \cap \mbox{Str}(I)\neq \emptyset}D\cap K.$$
In addition, for any $D \in \mathcal{D}_b,$ since $f_D$ is homothetic of Lipshitz constant $\vert D\vert,$ one has $$\mathcal{H}^s_{\infty}(D\cap K) =\mathcal{H}^s_{\infty}(f_D(K))=\vert D\vert^s \mathcal{H}^s_{\infty}(K)=\kappa_s \vert D\vert^s.$$ Thus, up to multiplying by some constant depending on $s$ if one must, we may consider only coverings by $b$-adic cubes of generation smaller than $n$.

Fix a covering $\mathcal{C}$ of $\mbox{Str}(I)$ by $b$-adic cubes of generation smaller than $n.$ There must exist $0 \leq k\leq n$ such that $$\mu_0\Big(\bigcup_{D \in\mathcal{C}: \ \mbox{gen}(D)=k}D\cap \mbox{Str(I)}\Big) \geq \frac{1}{n+1}\mu_0(R).$$
In addition, notice that,   for every $D_1,D_2 \in \left\{D \in\mathcal{C}: \ \mbox{gen}(D)=k\right\}$, one has $$f^{-1}_{D_1}(\mbox{Str}(I))=f^{-1}_{D_2}(\mbox{Str}(I)),$$
so that, by self-similarity,
\begin{align*}
\mu_0(D_1 \cap \mbox{Str}(I))&=\mu_0(D_1)\times \mu_0\Big(f^{-1}_{D_1}(\mbox{Str}(I)\Big)\\
&=b^{-ks_0}\mu_0\Big(f^{-1}_{D_1}(\mbox{Str}(I)\Big)=\mu_0(D_2)\times \mu_0\Big(f^{-1}_{D_2}(\mbox{Str}(I)\Big)=\mu_0(D_2 \cap \mbox{Str}(I)).
\end{align*}
Moreover, fixing $x\in D_1\cap K,$ one has $$\mu_0(D_1 \cap \mbox{Str}(I))=b^{-ks_0}\mu_0\Big(\mbox{Str}\Big(I_{n-k}(\pi_2(\sigma^k(x)))\Big)\Big),$$
which implies that $$\#\left\{D\in\mathcal{C}: \ \mbox{gen}(D)=k\right\}\geq \frac{\mu_0 (R)}{(n+1)b^{-ks_0}\mu_0\Big(\mbox{Str}\Big(I_{n-k}(\pi_2(\sigma^k(x)))\Big)\Big)}.$$

Thus

\begin{align*}
\sum_{D\in\mathcal{C}}\vert D \vert^s &\geq \sum_{D\in\mathcal{C}: \ \mbox{gen}(D)=k}\vert D \vert^s=b^{-ks} \#\left\{D \in\mathcal{C}: \ \mbox{gen}(D)=k\right\}\\
&\geq b^{-k(s-s_0)}\times \frac{1}{n+1}\times  \frac{\mu_0 (R)}{\mu_0\Big(\mbox{Str}\Big(I_{n-k}(\pi_2(\sigma^k(x)))\Big)\Big)}.
\end{align*}
Noticing that $n+1 \leq 2n=-2\log_b \vert I \vert$ and that the above inequality holds up to some multiplicative constants depending on $s$, the lower-bound is obtained by setting $k=\lfloor n \gamma \rfloor$, with $\frac{k}{n}\leq\gamma \leq \frac{k+1}{n}.$

The upper-bound is simply obtained by covering at the scale $k=\lfloor \gamma n\rfloor$ corresponding to the minimum and applying the same estimates.
\end{proof}
 Given $x\in K,$ Proposition \ref{MainPropoCont},  actually allows one to estimate $\mathcal{H}^s_{\infty}(R\cap K)$ for any rectangle of the form $J \times I$, where $I$ is $b$-adic interval and $J$ is any interval centered on $x$.   
\begin{corollary}
 Let $I$ be a $b$-adic interval and $J$ an interval such that $\frac{1}{b}J \cap K \neq \emptyset$ and $\frac{\vert J\vert}{b}\geq \vert I\vert.$ Let $D $ be the largest b-adic interval contained in $J$. Notice that $f_{D}^{-1}(R)$ is a stripe of the form $[0,1]\times \widetilde{I}$, where $\widetilde{I}$ is a $b$-adic interval. Moreover one has $$b^{-s\mbox{gen}(D)} \mathcal{H}^s_{\infty}(f_{D}^{-1}(R)\cap K)\leq\mathcal{H}^s_{\infty}(K\cap R)\leq 3\times b^{-s\mbox{gen}(D)} \mathcal{H}^s_{\infty}(f_{D}^{-1}(R)\cap K) $$ 
\end{corollary}
\begin{proof}
Simply notice that $D\cap R$ intersects $K$ and contains a full stripe. In addition, $R$  intersects at most two more dyadic cubes of the same generation as $D$, say $D_1$ and $D_2$, but by self-similarity $$\max_{i=1,2}\mathcal{H}^{s}_{\infty}(D_i \cap R \cap K)\leq \mathcal{H}^s_{\infty}(R\cap D\cap K)$$
so that $$\mathcal{H}^s_{\infty}(R\cap D\cap K)\leq \mathcal{H}^s_{\infty}(K\cap R)\leq 3\mathcal{H}^s_{\infty}(R\cap D\cap K).$$
In addition, since $f_D$ is homotethic, $$\mathcal{H}^s_{\infty}(f_D^{-1}(R)\cap K)=\vert D\vert^{-s}\mathcal{H}^s_{\infty}(D\times I \cap K).$$
\end{proof}

\begin{proposition}
\label{Propoascontent}
Let $\nu \in\mathcal{M}([0,1]^2)$ be a $\times b$ ergodic measure. Then, for $\nu$-almost every $x$ for every $\varepsilon>0$, there exists $r_x>0 $ such that for every $r\leq r_x$, for every $0\leq s\leq s_0,$ one has $$r^{\max\left\{s\tau_1, s\tau_2 -(\tau_2-\tau_1)(s_0 -\alpha_{\nu})\right\}+\varepsilon} \leq \mathcal{H}^{s}_{\infty}\Big((x+(-r^{\tau_1},r^{\tau_1})\times(-r^{\tau_2},r^{\tau_2})) \cap K  \Big) \leq r^{\max\left\{s\tau_1, s\tau_2 -(\tau_2-\tau_1)(s_0 -\alpha_{\nu})\right\}-\varepsilon} $$
\end{proposition}
\begin{proof}
Fix $\varepsilon_0 >0.$ Recall that, since $\pi_2 \nu$ is $\times b$ ergodic, either $\nu$ is a periodic measure (hence is carried by finitely many atoms) or $\pi_2 \nu$ is diffuse, which case, by Corollary \ref{Corodya}, there exists $r_1>0$ small enough so that for every $r\leq r_1,$ one has  $$ D_{\lfloor \frac{-\tau_1\log r}{\log b}\rfloor}(\pi_2 (x))\subset B\Big(\pi_2 (x) ,r^{\tau_1}\Big)\subset  D_{\lfloor \frac{-\tau_1(1-\varepsilon_0)\log r}{\log b}\rfloor}(\pi_2 (x)).$$
Thus, writing $\tau =\frac{\tau_2}{\tau_1},$ and $$R_{n,\tau}(x)=(\pi_1(x)-b^{n},\pi_1(x)+b^{n})\times I_{\lfloor n\tau \rfloor }(\pi_2(x)),$$ it is enough to show that for every large enough $n$, one has $$b^{-n \max\left\{s, s\tau -(\tau-1)(s_0 -\alpha_{\nu})\right\}+\varepsilon} \leq \mathcal{H}^{s}_{\infty}\Big(R_{n,\tau}(x) \cap K  \Big) \leq b^{-n\max\left\{s, s\tau -(\tau-1)(s_0 -\alpha_{\nu})\right\}-\varepsilon}.$$
Notice that $R_{n,\tau}(x)$ can be covered by at most $3$, from left to right, say $D_1,D_2:=D_{n}(x),D_3$ $b$-adic cubes of generation $n$. Writing $f_{D_1},f_{D_2}, f_{D_3}$ the corresponding contractions, one has $$f_{D_1}^{-1}(R_{n,\tau}(x)),f_{D_3}^{-1}(R_{n,\tau}(x))\subset f_{D_2}^{-1}(R_{n,\tau}(x)).$$
This yields that $$\mathcal{H}^{s}_{\infty}\Big(R_{n,\tau}(x)\cap K \cap D_{n}(x)\Big)\leq \mathcal{H}^{s}_{\infty}\Big(R_{n,\tau}(x)\Big)\leq 3 \mathcal{H}^{s}_{\infty}\Big(R_{n,\tau}(x)\cap K \cap D_{n}(x)\Big).$$
Also, $$R_{n,\tau}(x) \cap D_n(x)=I_{n}(\pi_1 (x))\times I_{\lfloor n\tau\rfloor}(\pi_2(x)).$$
By Proposition \ref{PropLocDimCenter}, there exists $n_1 \in\mathbb{N}$ such that, setting $$A_{n_1}=\left\{y \in [0,1]^2: \ \forall k\geq n_1 \ b^{-k(\alpha_{\nu}+\varepsilon)}\leq \pi_2 \mu_0 (I_{k}(\pi_2(y)))\leq b^{-k(\alpha_{\nu}-\varepsilon)}\right\} $$ 
satisfies that $$\nu(A_{n_1})\geq 1-\varepsilon_0.$$
By ergodicity, for $\nu$-almost every $x$, there exists $n_3$ large enough so that for every $k\geq n_3,$ there exists $(1-2\varepsilon_0)k\leq k_1 \leq k \leq k_2 \leq (1+2\varepsilon_0)k$ for which $b^{k_i}x \in A_{n_1},$ for $i=1,2.$ Thus provided that $n$ is large enough, fix  $(1-2\varepsilon_0)n\leq n_1 \leq n \leq n_2 \leq (1+2\varepsilon_0)n$ and remark that $$I_{n_2}(\pi_1(x))\times I_{\lfloor n\tau \rfloor}(\pi_2(x))\subset I_{n}(\pi_1(x))\times I_{\lfloor n\tau \rfloor}(\pi_2(x)) \subset I_{n_1}(\pi_1(x))\times I_{\lfloor n\tau \rfloor}(\pi_2(x)).$$
Since $f_{D_{n_1}(x)},f_{D_{n_2}(x)}$ are homothetic, we have $$\mathcal{H}^{s}_{\infty}\Big(f_{D_{n_i}(x)}^{-1}\Big( I_{n_i}(\pi_1(x))\times I_{\lfloor n\tau \rfloor}(\pi_2(x))\cap K\Big)\Big)=\vert D_i \vert^{-s}\mathcal{H}^{s}_{\infty}\Big( I_{n_i}(\pi_1(x))\times I_{\lfloor n\tau \rfloor}(\pi_2(x))\cap K\Big).$$
Moreover, remark that 
$$f_{D_{n_i}(x)}^{-1}\Big( I_{n_i}(\pi_1(x))\times I_{\lfloor n\tau \rfloor}(\pi_2(x))\Big)=[0,1]\times I_{\lfloor n\tau \rfloor -n_i}(\pi_2(b^{n_i}(x))).$$
Applying Proposition \ref{MainPropoCont} yields $$\frac{C}{\log(n)}\leq \frac{ \mathcal{H}^{s}_{\infty}\Big(f_{D_{n_i}(x)}^{-1}\Big( I_{n_i}(\pi_1(x))\times I_{\lfloor n\tau \rfloor}(\pi_2(x))\Big)\cap K\Big)}{ \min_{0\leq \gamma \leq 1}b^{-(s-s_0)\lfloor \gamma( \lfloor n\tau \rfloor -n_i)\rfloor}\times \pi_2\mu_0(I_{\lfloor \gamma(\lfloor n\tau \rfloor -n_i) \rfloor}(\pi_2(b^{n_i}(x)))) } \leq 1.$$
Noticing that $(\tau-1 -2\varepsilon_0)n\leq n\tau -n_i \leq (\tau-1 +2\varepsilon_0)n$, one has 
$$\begin{cases}\frac{\min_{0\leq \gamma \leq 1}b^{-(s-s_0)\lfloor \gamma( \lfloor n\tau \rfloor -n_i)\rfloor}\times \pi_2\mu_0(I_{\lfloor \gamma(\lfloor n\tau \rfloor -n_i) \rfloor}(\pi_2(b^{n_i}(x))))}{ \min_{k \geq (\tau-1 -2\varepsilon_0)n}b^{-(s-s_0)k}\times \pi_2\mu_0(D_{k}(\pi_2(b^{n_i}(x))))} \leq 1\\  \frac{\min_{0\leq \gamma \leq 1}b^{-(s-s_0)\lfloor \gamma( \lfloor n\tau \rfloor -n_i)\rfloor}\times \pi_2\mu_0(I_{\lfloor \gamma(\lfloor n\tau \rfloor -n_i) \rfloor}(\pi_2(b^{n_i}(x))))}{ \min_{k \geq (\tau-1 +2\varepsilon_0)n}b^{-(s-s_0)k}\times \pi_2\mu_0(I_{k}(\pi_2(b^{n_i}(x))))} \geq 1.\end{cases}  $$
This yields that, for some $\kappa >0,$ 
$$ b^{-\kappa \varepsilon_0 n}\leq\frac{\min_{0\leq \gamma \leq 1}b^{-(s-s_0)\lfloor \gamma( \lfloor n\tau \rfloor -n_i)\rfloor}\times \pi_2\mu_0(D_{\lfloor \gamma(\lfloor n\tau \rfloor -n_i) \rfloor}(\pi_2(b^{n_i}(x))))}{ \min_{k \geq (\tau-1 )n}b^{-(s-s_0)k}\times \pi_2\mu_0(I_{k}(\pi_2(b^{n_i}(x))))}\leq b^{\kappa \varepsilon_0 n}.$$
There exists a constant $\beta_{n_1}$, depending only on $n_1$, such that such that  $$\beta_{n_1}\leq\min_{k \geq n_1}b^{-(s-s_0)k}\times \pi_2\mu_0(I_{k}(\pi_2(b^{n_i}(x)))) \leq 1$$
and, since $b^{n_i}(x)\in A_{n_1},$ for every $ (\tau-1 )n \leq k \leq n_1,$ one has $$b^{-k(\alpha_{\nu}+\varepsilon_0)}\leq\pi_2\mu_0(I_{k}(\pi_2(b^{n_i}(x))))\leq b^{-k(\alpha_{\nu}-\varepsilon_0)}.$$
This yields 
\begin{align*}
 \min_{ (\tau-1 )n \leq k \leq n_1} b^{-k(s-s_0 +\alpha_{\nu}+\varepsilon_0)}
&\leq\min_{ (\tau-1 )n \leq k \leq n_1}b^{-(s-s_0)k}\times \pi_2\mu_0(I_{k}(\pi_2(b^{n_i}(x))))  \\
 \min_{ (\tau-1 )n \leq k \leq n_1} b^{-k(s-s_0 +\alpha_{\nu}-\varepsilon_0)}&\geq \min_{ (\tau-1 )n \leq k \leq n_1}b^{-(s-s_0)k}\times \pi_2\mu_0(I_{k}(\pi_2(b^{n_i}(x)))).
\end{align*}
Hence, for every $s\leq s_0 -\alpha_{\nu}-\varepsilon_0$ or  $s\geq s_0 -\alpha_{\nu}-\varepsilon_0$ since, $k\mapsto k(s-s_0 +\alpha_{\nu}-\varepsilon_0)$ and $k\mapsto k(s-s_0 +\alpha_{\nu}+\varepsilon_0)$ are monotonic the infimimum is obtained for $k=\lfloor (\tau-1 )n \rfloor +1$ or $k=n_1$ and there exists a constant $\theta_{n_1}$ such that

\begin{align*}
 \min\left\{\beta_{n_1},b^{-(\tau-1)n(s-s_0 +\alpha_{\nu}+\varepsilon_0)}\right\}
&\leq\min_{ (\tau-1 )n \leq k \leq n_1}b^{-(s-s_0)k}\times \pi_2\mu_0(I_{k}(\pi_2(b^{n_i}(x))))  \\
 \min_{ (\tau-1 )n \leq k \leq n_1}b^{-(s-s_0)k}\times \pi_2\mu_0(I_{k}(\pi_2(b^{n_i}(x)))) &\leq \min\left\{\beta_{n_1},b^{-(\tau-1)n(s-s_0 +\alpha_{\nu}-\varepsilon_0)}\right\}.
\end{align*}
Finally, since $s\mapsto \mathcal{H}^{s}_{\infty}(\cdot)$ is non increasing, for every $s\leq s_0,$  (in particular for $ s_0 -\alpha_{\nu}-\varepsilon_0\leq s\leq s_0 -\alpha_{\nu}+\varepsilon_0$), there exists $\omega_{n_1}>0 $ such that one has 

\begin{align*}
 \frac{C}{\log n}\min\left\{\omega_{n_1},b^{-(\tau-1)n(s-s_0 +\alpha_{\nu}+3\varepsilon_0)}\right\}&\leq\mathcal{H}^{s}_{\infty}\Big(f_{D_{n_i}(x)}^{-1}\Big( I_{n_i}(\pi_1(x))\times I_{\lfloor n\tau \rfloor}(\pi_2(x))\Big)\cap K\Big) \\  \min\left\{\omega_{n_1},b^{-(\tau-1)n(s-s_0 +\alpha_{\nu}-3\varepsilon_0)}\right\}&\geq \mathcal{H}^{s}_{\infty}\Big(f_{D_{n_i}(x)}^{-1}\Big( I_{n_i}(\pi_1(x))\times I_{\lfloor n\tau \rfloor}(\pi_2(x))\Big)\cap K\Big), 
\end{align*}
 which implies that 

\begin{align*}
\frac{C}{\log n}\min\left\{b^{-(1+\varepsilon_0)n s}\omega_{n_1},b^{-(\tau-1)n(s-s_0 +\alpha_{\nu}+3\varepsilon_0)-(1+\varepsilon_0)ns}\right\}&\leq\mathcal{H}^{s}_{\infty}\Big( I_{n_i}(\pi_1(x))\times I_{\lfloor n\tau \rfloor}(\pi_2(x))\cap K\Big)\\ \min\left\{b^{-(1-\varepsilon_0)n s}\omega_{n_1},b^{-(\tau-1)n(s-s_0 +\alpha_{\nu}+3\varepsilon_0)-(1-\varepsilon_0)ns}\right\}&\geq \mathcal{H}^{s}_{\infty}\Big( I_{n_i}(\pi_1(x))\times I_{\lfloor n\tau \rfloor}(\pi_2(x))\cap K\Big).
\end{align*}
Recalling the inclusion and since $\log(n)=b^{o(n)}$, provided that $n$ is large enough, one gets 

\begin{align*}
\min\left\{b^{-(1+4\varepsilon_0)n s},b^{-(\tau-1)n(s-s_0 +\alpha_{\nu}+4\varepsilon_0)-(1+4\varepsilon_0)ns}\right\}&\leq \mathcal{H}^{s}_{\infty}(R_{n,\tau}(x)\cap K) \\  \min\left\{b^{-(1-4\varepsilon_0)n s},b^{-(\tau-1)n(s-s_0 +\alpha_{\nu}-4\varepsilon_0)-(1-4\varepsilon_0)ns}\right\}&\geq \mathcal{H}^{s}_{\infty}(R_{n,\tau}(x)\cap K).
\end{align*}
Moreover \begin{align*}
\min\left\{b^{-(1+4\varepsilon_0)n s},b^{-(\tau-1)n(s-s_0 +\alpha_{\nu}+4\varepsilon_0)-(1+4\varepsilon_0)ns}\right\}&\geq b^{-n \max\left\{s,s\tau-(\tau-1)(s_0 -\alpha)\right\} +(4s_0 +4\tau)\varepsilon_0 } \\
\min\left\{b^{-(1-4\varepsilon_0)n s},b^{-(\tau-1)n(s-s_0 +\alpha_{\nu}-4\varepsilon_0)-(1-4\varepsilon_0)ns}\right\}&\leq b^{-n \max\left\{s,s\tau-(\tau-1)(s_0 -\alpha)\right\} -(4s_0 +4\tau)\varepsilon_0 }.
\end{align*}

One concludes by taking $\varepsilon_0$ so small that $(4s_0 +4\tau)\varepsilon_0 \leq \varepsilon.$

\end{proof}

\section{Proof of Theorem \ref{ThmTree} and Theorem \ref{ThmRand}}
\subsection{Proof of Theorem \ref{ThmTree}}
In this section again, given $x\in [0,1]$ and $n\in\mathbb{N}$, we will write $I_n(x)$ the $b$-adic interval of generation $n$ containing $x$ and given $y\in [0,1]^2,$ $D_n(y)$ still denotes the $b$-adic cube of generation $n$ containing $y$. We fix $S$ , $\mu$, $1\leq \tau_1 <\tau_2$  as in Theorem \ref{ThmTree} and write $\tau=\frac{\tau_2}{\tau_1}$.

\bigskip

We start by dealing with the case where $\mu_0$ is supported on an horizontal or vertical line. 

Assume that $\mu_0$ is supported on an horizontal line. To prove that Theorem \ref{ThmTree} provides the correct estimate, one simply needs to prove that one recovers the same results as the case of of balls i.e. (see \cite{ED4}) that 
$$\dim_H V_{\tau_1,\tau_2}(x)\geq \frac{\dim_H \mu}{\tau_1}.$$

In this case $ \pi_2 \mu_0 $ is an atom so that for any $\nu $ ergodic,  $\alpha_{\nu}=0.$ Moreover
\begin{align*}
\frac{\dim_H \mu}{\tau_1} &\leq \frac{\dim_H \mu +(\tau_2 -\tau_1) s_0}{\tau_2} \\
\Leftrightarrow \frac{\dim_H \mu}{\tau_1}& \leq \frac{\dim_H \mu }{ \tau \tau_1}+(1-\frac{1}{\tau})s_0 \\
\Leftrightarrow\frac{\dim_H \mu}{\tau_1}&\leq s_0 
\end{align*}
which is always satisfied as $\tau_1 \geq \frac{1}{s_0}$ and $\dim_H \mu \leq 1.$ Thus one recovers the correct estimate in that case.

The case where $\mu_0$ is carried by a vertical line is straightforward as $\dim_H \pi \mu_0 =s_0$ and $\frac{1}{\tau_2}\leq \frac{1}{\tau_1}.$

\bigskip

We now assume that $\mu_0$ is not carried by an horizontal or vertical line and we use the notations of Subsection \ref{SecContEsti}. Before proving Theorem \ref{ThmTree}, we establish some estimates.  Let us fix $\nu$ an ergodic measure with respect to $S$ and $\alpha_{\nu}$ such that for $\nu$-almost every $x$, $$\dim(\pi_2(x),\pi_2 \mu_0)=\alpha_{\nu} .$$
In what follows, given a word $\underline{i}=(i_1,...,i_n)\in\left\{1,...,m\right\}^n,$ we will write $$\begin{cases}R_{\underline{i}}(x)=f_{\underline{i}}(x)+(-\vert f_{\underline{i}}(K)\vert^{\tau_1},\vert f_{\underline{i}}(K)\vert^{\tau_1})\times (-\vert f_{\underline{i}}(K)\vert^{\tau_2},\vert f_{\underline{i}}(K)\vert^{\tau_2})\\
D_{\underline{i}}=f_{\underline{i}}\Big([0,1)^2).\end{cases} $$

Let us fix $n\in\mathbb{N}$ and $(i_1,...,i_n)\in\left\{1,...,m\right\}^n.$ Notice that, for any $k\leq n$,  $$f_{(i_1,...,i_{n+k})}^{-1}(R_{\underline{i}})(x)=b^{k}(x)+\Big(-b^{-n(\tau_1 -1 )+k}, b^{-n(\tau_1 -1 )+k}\Big)\times \Big(-b^{-n(\tau_2 -1 )+k}, b^{-n(\tau_2 -1 )+k}\Big).$$

Fix $\varepsilon,\varepsilon_0>0,$ By Proposition \ref{Propoascontent} applied simultaneously  with $\tau_1 =\tau_1 -1 ,\tau_2 =\tau_2 -1$ and $\tau_1 =\tau_1 -1-\varepsilon_0, $ $\tau_2=\tau_2-1-\varepsilon_0$, there exists $\rho>0$ such that, writing $$A_{\rho}=\left\{x \in K :\forall r\leq \rho, \begin{cases}r^{\varepsilon} \leq \frac{\mathcal{H}^{s}_{\infty}\Big((x+(-r^{\tau_1 -1},r^{\tau_1 -1})\times(-r^{\tau_2 -1},r^{\tau_2 -1})) \cap K  \Big)}{r^{\max\left\{s(\tau_1-1), s(\tau_2-1) -(\tau_2-\tau_1)(s_0 -\alpha_{\nu})\right\}} } \\ r^{-\varepsilon} \geq \frac{\mathcal{H}^{s}_{\infty}\Big((x+(-r^{\tau_1 -1 -\varepsilon_0},r^{\tau_1 -1 -\varepsilon_0})\times(-r^{\tau_2 -1-\varepsilon_0},r^{\tau_2 -1-\varepsilon_0})) \cap K  \Big)}{r^{\max\left\{s(\tau_1-1), s(\tau_2-1) -(\tau_2-\tau_1)(s_0 -\alpha_{\nu})\right\}} } \end{cases}\right\}, $$
one has $\nu(A_{\rho})\geq 1-\varepsilon_0.$ This implies, by ergodicity, that that for $\nu$-almost every $x$, there exists $k \in\mathbb{N}$ so large that for every $b^{k}x \in A_{\rho}.$ 
\begin{align*}
&b^{k}x+(-b^{-n(\tau_1 -1)},b^{-n(\tau_1 -1)})\times(-b^{-n(\tau_2 -1)},b^{-n(\tau_2 -1)})\\
& \subset b^{k}x+\Big(-b^{-n(\tau_1 -1 )+k}, b^{-n(\tau_1 -1 )+k}\Big)\times \Big(-b^{-n(\tau_2 -1 )+k},b^{-n(\tau_2 -1 )+k}\Big)\\
&\text{ and } \\ &b^{k}x+(-b^{-n(\tau_1 -1)+k},b^{-n(\tau_1 -1)+k})\times(-b^{-n(\tau_2 -1)},b^{-n(\tau_2 -1)+k})\\
&\subset b^{k}x+(-b^{-n(\tau_1 -1 -\varepsilon_0)},b^{-n(\tau_1 -1 -\varepsilon_0)})\times(-b^{-n(\tau_2 -1-\varepsilon_0)},b^{-n(\tau_2 -1-\varepsilon_0)}).
\end{align*}
Since $$\mathcal{H}^{s}_{\infty}(f_{(i_1,...,i_{n+k})}^{-1}(R_{\underline{i}})\cap K)=b^{-(n+k)s}\times \mathcal{H}^{s}_{\infty}(R_{\underline{i}}\cap K) $$
and, for every large enough $n$, one has $$  b^{-n\varepsilon}\leq\frac{\mathcal{H}^{s}_{\infty}\Big(f_{(i_1,...,i_{n+k})}^{-1}(R_{\underline{i}})\cap K \Big)}{b^{-n\max\left\{s(\tau_1-1), s(\tau_2-1) -(\tau_2-\tau_1)(s_0 -\alpha_{\nu})\right\}} }\leq b^{n\varepsilon},$$
one gets that, provided that $n$ is large enough,
\begin{align*}
b^{-2\varepsilon n}\leq\frac{\mathcal{H}^{s}_{\infty}(R_{\underline{i}}\cap K)}{b^{-ns} \times b^{-n\max\left\{s(\tau_1-1), s(\tau_2-1) -(\tau_2-\tau_1)(s_0 -\alpha_{\nu})\right\}} }\leq b^{2n\varepsilon}
\end{align*}
so that

\begin{align*}
b^{-2\varepsilon n}\leq\frac{\mathcal{H}^{s}_{\infty}(R_{\underline{i}}\cap K)}{ b^{-n\max\left\{\tau_1, s\tau_2 -(\tau_2-\tau_1)(s_0 -\alpha_{\nu})\right\}} }\leq b^{2n\varepsilon}.
\end{align*}
 
We are now ready to prove Theorem \ref{ThmTree}. Let $$\Lambda_{\mu}\subset \bigcup_{k \geq 1}\left\{1,...,m\right\}^k $$ be a set of words such that there exists $z \in K$ for which $$ \mu\Big(\limsup_{\underline{i}\in \Lambda_{\mu}}B\Big(f_{i}(z),\vert f_{\underline{i}}(K)\vert\Big)\Big)=1 .$$
 Since for  any $x\in K,$ $$B\Big(f_{i}(z),\vert f_{\underline{i}}(K)\vert\Big)\Big)\subset B\Big(f_{i}(x),2\vert f_{\underline{i}}(K)\vert\Big)\Big),$$ 
one has $$\mu\Big(\limsup_{\underline{i}\in \Lambda_{\mu}}B\Big(f_{i}(x),2\vert f_{\underline{i}}(K)\vert\Big)\Big)=1.$$

In this, case (see \cite{ed2}), for any $\varepsilon>0$ there exists a subset of words $\widetilde{\Lambda}_{\mu}\subset \Lambda_{\mu}$ such that $$\begin{cases}\forall \underline{i}\in \widetilde{\Lambda}_{\mu}, \ \mu\Big(B\Big(f_{i}(x),2\vert f_{\underline{i}}(K)\vert\Big)\Big)\leq b^{-\vert \underline{i}\vert(\dim_H \mu-\varepsilon)} \\ \mu\Big(\limsup_{\underline{i}\in \widetilde{\Lambda}_{\mu}}B\Big(f_{i}(x),2\vert f_{\underline{i}}(K)\vert\Big)\Big)=1 \end{cases} $$

If $s_{\varepsilon}$ is solution to the equation $$\max\left\{s\tau_1, s\tau_2 -(\tau_2-\tau_1)(s_0 -\alpha_{\nu})\right\}+\varepsilon = \dim_H \mu -\varepsilon,$$
for every $\underline{i}\in\widetilde{\Lambda}_{\mu}$ of large enough generation,  $$\begin{cases}\mathcal{H}^{s_{\varepsilon}}_{\infty}(R_{\underline{i}}\cap K)\geq \mu\Big(B\Big(f_{i}(x),2\vert f_{\underline{i}}(K)\vert\Big)\Big) \\  \mu\Big(\limsup_{\underline{i}\in \widetilde{\Lambda}_{\mu}}B\Big(f_{i}(x),2\vert f_{\underline{i}}(K)\vert\Big)\Big)=1.
\end{cases} $$
Thus Theorem \ref{MTPss}, one gets $$\dim_H \limsup_{ \underline{i}\in\Lambda_{\mu}}R_{\underline{i}}\geq s_{\varepsilon}.$$
Letting $\varepsilon\to 0,$ yields 

 $$\dim_H \limsup_{ \underline{i}\in\Lambda_{\mu}}R_{\underline{i}}\geq s$$
where $s$ is solution to  $$\max\left\{s\tau_1, s\tau_2 -(\tau_2-\tau_1)(s_0 -\alpha_{\nu})\right\} = \dim_H \mu ,$$
i.e., $$s=\min\left\{\frac{\dim_H \mu}{\tau_1}, \frac{\dim_H \mu +(\tau_2 -\tau_1)(s_0 -\alpha_{\nu})}{\tau_2}\right\}.$$
Assume in addition that $$\lim_{\vert\underline{i}\vert \to +\infty} \frac{-\log \mu\Big(f_{\underline{i}}([0,1]^2))\Big)}{\vert \underline{i}\vert \log b} =  \dim_H \mu,$$
Then, for any $\underline{i}\in\Lambda_{\mu},$ and any $\varepsilon>0,$ there exists $N\in\mathbb{N}$ such that for every $n\geq N,$ every $\underline{i}\in \Lambda_{\mu}\cap \left\{1,...,m\right\}^n,$ writing $t_{\varepsilon}$ the solution to
$$\max\left\{s\tau_1, s\tau_2 -(\tau_2-\tau_1)(s_0 -\alpha_{\nu})\right\}-\varepsilon = \dim_H \mu +\varepsilon,$$

 one has $$\mathcal{H}^{t_{\varepsilon}}_{\infty}(R_{\underline{i}}(x)\cap K)\leq b^{-n\varepsilon}\mu\Big(f_{\underline{i}}([0,1]^d)\Big).$$
And 
\begin{align*}
\sum_{n\geq 1} b^{-n \varepsilon}\sum_{\underline{i}\in\left\{1,...,m\right\}^n}\mu\Big(f_{\underline{i}}([0,1]^d)\Big)=\sum_{n\geq 1} b^{-n \varepsilon}<+\infty.
\end{align*}
This yields in particular $$\sum_{\underline{i}\in\Lambda_{\mu}}\mathcal{H}^{t_{\varepsilon}}_{\infty}(R_{\underline{i}}(x)\cap K)<+\infty,$$
so that $\dim_H \limsup_{\underline{i}\in\Lambda_{\mu}} R_{\underline{i}}(x)\leq t_{\varepsilon}.$ 

Letting $\varepsilon \to 0,$ one gets $$\dim_H \limsup_{ \underline{i}\in\Lambda_{\mu}}R_{\underline{i}}\leq s,$$
which finishes the proof.


\subsection{Proof of Theorem \ref{ThmRand}}

Let us fix $m,b\in\mathbb{N}$, $S=\left\{f_1,...,f_m\right\}$ a two-dimensional base $b$ missing digit IFS, $\mu_0$ the corresponding Alfhors-regular measure and $s_0$ its dimension.

 One will only prove Theorem \ref{ThmRand} in the case where $(X_n)_{n\in\mathbb{N}}$ is an i.i.d. sequence of law $\mu_0.$ In Section \ref{SectionDyna}, we will explain how one recovers the result regards $\mu_0$-typical orbits by adapting a lemma established as\cite[Lemma 7.3]{Edergo}.
 
\medskip

First, we justify, as in the previous section that, in the case where $\mu_0$ is carried by an horizontal or vertical line, Theorem \ref{ThmRand} provides the same bound as in the case of random balls, thus is correct. In the case where $\mu_0$ is carried by an horizontal line, one has $\dim_H \pi_2 \mu_0 =0,$ so that one always as $\frac{1}{\tau_1}\leq s_0 -\dim_H \pi_2 \mu_0$ and $\dim_H W_{\tau_1, \tau_2}=\frac{1}{\tau_1}.$ 

\medskip

In the case where $\mu_0$ is carried by an horizontal line, $\dim_H \pi_2 \mu_0 =s_0, $ so that $\kappa_2 =s_0$ and $v_{\tau}(\kappa_2)=0$, which yields $$\dim_H W_{\tau_1,\tau_2}=\frac{1+(\tau_2-\tau_1)\times 0}{\tau_2}=\frac{1}{\tau_2},$$ 
and the conclusion of Theorem \ref{ThmRand} holds true in that case.

\medskip

We now focus on the case where $\mu_0$ is not carried bay an horizontal or vertical line and we use the same notations as in Section \ref{SecContEsti}.

 Given $\tau_1,\tau_2,r>0$ and $x=(x_1,x_2)\in \mathbb{R}^2$, let us write $$R_{\tau_1,\tau_2}(x,r)=x+(-r^{\tau_1},r^{\tau_1})\times (-r^{\tau_2},r^{\tau_2}) .$$

Given $\alpha\in \mbox{Spectr}(\pi_2 \mu_0),$ define $s_{\alpha}$ as the solution to  $$\max\left\{s,s\tau -(\tau-1)(s_0 -\alpha)\right\}=-(\tau-1)(\alpha_0-D_{\pi_{2}(\mu_0)}(\alpha)) +\frac{1}{\tau_1},$$
i.e. 
 $$s_{\alpha}=\min\left\{\frac{1-(\tau_2-\tau_1)(\alpha -D_{\pi_2 \mu_0}(\alpha))}{\tau_1},\frac{1+(\tau_2-\tau_1)(s_0 -2\alpha +D_{\pi_2 \mu_0}(\alpha))}{\tau_2}\right\}.$$
 
In the next sections, we show the following. 
 
\begin{theoreme}
\label{ThmRandbis}
Let $(X_n)_{n\in\mathbb{N}}$ be an i.i.d. sequence of common law $\mu_0$ and $\frac{1}{s_0}\leq\tau_1\leq \tau_2.$

Almost surely, $$\dim_H \limsup_{n\to +\infty}R_{\tau_1,\tau_2}(X_n,\frac{1}{n})=\max_{\alpha \in \mbox{Spectr}(\pi_2 \mu_0)}s_{\alpha}.$$
\end{theoreme}

We justify below  that Theorem \ref{ThmRandbis} implies Theorem \ref{ThmRand}. 

Given $\alpha \in\mbox{ Spectr }(\pi_2 \mu_0),$ notice that, writing again $\tau =\frac{\tau_2}{\tau_1},$ and assuming $\tau >1$,
\begin{align*}
&\frac{1-(\tau_2-\tau_1)(\alpha -D_{\pi_2 \mu_0}(\alpha))}{\tau_1}\leq \frac{1+(\tau_2-\tau_1)(s_0 -2\alpha +D_{\pi_2 \mu_0}(\alpha))}{\tau_2} \\
&\Leftrightarrow  \frac{1}{\tau_1}-(\tau-1)(\alpha -D_{\pi_2 \mu_0}(\alpha))\leq \frac{\frac{1}{\tau_1}+(\tau-1)(s_0 -2\alpha +D_{\pi_2 \mu_0}(\alpha))}{\tau} \\ 
&\Leftrightarrow (\tau-1)\frac{1}{\tau_1}-\tau(\tau-1)(\alpha -D_{\pi_2 \mu_0}(\alpha))\leq (\tau-1)(s_0 -2\alpha +D_{\pi_2 \mu_0}(\alpha))\\
&\Leftrightarrow \frac{1}{\tau_1}\leq s_0 +(\tau-2)\alpha -(\tau-1)D_{\pi_2 \mu_0}(\alpha)\\
&\Leftrightarrow \frac{1}{\tau_1}\leq v_{\tau}(\alpha).
\end{align*}

We consider $3$ cases separately.

\bigskip

\textbf{Case 1: $\frac{1}{\tau} \leq v_{\tau}(\dim_H \pi_2(\mu_0))=s_0 -\dim_H \pi_2 \mu_0.$}

\bigskip

Notice that $\alpha \mapsto v_{\tau}(\alpha)$ is convex, since $\alpha \mapsto D_{\pi_2 \mu_0}(\alpha)$ is concave and its minimum is attained in $\alpha$ such that $$ D'_{\pi_2 \mu_0}(\alpha) =\frac{\tau-2}{\tau-1} \leq 1,$$
i.e. $\alpha =\kappa_{\frac{\tau-2}{\tau-1} }.$
As $\alpha \mapsto D'_{\pi_2 \mu_0}(\alpha)$ is non increasing and $D'_{\pi_2 \mu_0}(\dim_H \pi_2 \mu_0)=1,$ one has $$\kappa_2 \leq \dim_H \pi_2 \mu_0 \leq \kappa_{\frac{\tau-2}{\tau-1} }.$$

Moreover $\alpha \mapsto \frac{1-(\tau_2-\tau_1)(\alpha -D_{\pi_2 \mu_0}(\alpha))}{\tau_1}$ is non decreasing on $(-\infty,\dim_H \pi_2 \mu_0],$ non increasing on $[\mu_0,+\infty)$ and reaches its maximum $M_1$ on $\dim_H \pi_2 \mu_0$ with $M_1=\frac{1}{\tau_1}.$

Similarly, $\alpha \mapsto \frac{1+(\tau_2-\tau_1)(s_0 -2\alpha +D_{\pi_2 \mu_0}(\alpha))}{\tau_2}$ is concave, non decreasing on $(-\infty,\kappa_2],$ non increasing on $[\kappa_2,+\infty)$ and reaches its maximum $M_2$ on $\kappa_2$ with  $$M_2=  \frac{1+(\tau_2-\tau_1)(s_0 -2\kappa_2 +D_{\pi_2 \mu_0}(\kappa_2))}{\tau_2}.$$
This yields $v_{\tau}(\alpha)\geq \frac{1}{\tau_1}$ for any $\alpha \leq \dim_H \pi_2 \mu_0,$ so that $$\sup_{\alpha \in \mbox{ Spectr}(\pi_2 \mu_0)}s_{\alpha}=M_1 =\frac{1}{\tau_1}.$$

\bigskip

\textbf{Case 2:  $v_{\tau}(\dim_H \pi_2(\mu_0))\leq \frac{1}{\tau_1}\leq v_{\tau}(\kappa_2).$}

\bigskip

Call $\beta_{\tau_1,\tau_2}$ the solution to $$v_{\tau}(\beta_{\tau_1,\tau_2})=\frac{1}{\tau_1}.$$
The mapping $\alpha\mapsto \frac{1+(\tau_2-\tau_1)(s_0 -2\alpha +D_{\pi_2 \mu_0}(\alpha))}{\tau_2}$ is non increasing on $[\kappa_2,\beta_{\tau_1,\tau_2}]$ while $\alpha \mapsto  \frac{1-(\tau_2-\tau_1)(\alpha -D_{\pi_2 \mu_0}(\alpha))}{\tau_1}$ is non decreasing on $[\beta_{\tau_1,\tau_2},+\infty).$ We conclude in that case that 
\begin{align*}
\sup_{\alpha \in \mbox{ Spectr}(\pi_2 \mu_0)}s_{\alpha}&=\frac{1-(\tau_2-\tau_1)(\beta_{\tau_1,\tau_2} -D_{\pi_2 \mu_0}(\beta_{\tau_1,\tau_2}))}{\tau_1}\\
&= \frac{1+(\tau_2-\tau_1)(s_0 -2\beta_{\tau_1,\tau_2} +D_{\pi_2 \mu_0}(\beta_{\tau_1,\tau_2}))}{\tau_2}
\end{align*}

\bigskip

\textbf{Case 3: $\frac{1}{\tau_1}\geq v_{\tau}(\kappa_2 ).$ }

\bigskip

As both $\alpha\mapsto \frac{1+(\tau_2-\tau_1)(s_0 -2\alpha +D_{\pi_2 \mu_0}(\alpha))}{\tau_2}$  and  $\alpha \mapsto \frac{1+(\tau_2-\tau_1)(s_0 -2\alpha +D_{\pi_2 \mu_0}(\alpha))}{\tau_2}$ are non decreasing on $(-\infty,\kappa_2)$ and $\alpha\mapsto \frac{1+(\tau_2-\tau_1)(s_0 -2\alpha +D_{\pi_2 \mu_0}(\alpha))}{\tau_2}$ is non increasing on $[\kappa_2,+\infty),$ $s_{\alpha}$ reaches its maximum on $\kappa_2$ so that 

\begin{align*}
\sup_{\alpha \in \mbox{ Spectr}(\pi_2 \mu_0)}s_{\alpha}&= \frac{1+(\tau_2-\tau_1)(s_0 -2\kappa_2 +D_{\pi_2 \mu_0}(\kappa_2))}{\tau_2}
\end{align*}

\subsubsection{Proof of the upper-bound}
Write $$R_k = R_{\tau_1,\tau_2}(X_k,\frac{1}{k})$$

Our strategy to establish that, almost surely, one has $$\dim_H \limsup_{k\to +\infty}R_k \leq \sup_{\alpha \in \mbox{Spectr}(\pi_2 \mu_0)}s_{\alpha}$$
is to show that, for every $\varepsilon>0,$ writing $s_{\varepsilon}=\sup_{\alpha \in \mbox{Spectr}(\pi_2 \mu_0)}s_{\alpha}+\varepsilon,$ one has $$\sum_{k\geq 1} \mathcal{H}^{s_{\varepsilon}}_{\infty}\Big(R_k \cap K\Big)<+\infty.$$

Let us fix $\varepsilon>0.$ We recall that $\alpha\mapsto D_{\pi_2(\mu)}(\alpha)$ is continuous. Thus, by taking $\varepsilon$ smaller if one must, one may assume that $\vert \alpha-\alpha' \vert \leq \varepsilon \Rightarrow \vert D_{\pi_2(\mu)}(\alpha)-D_{\pi_2(\mu)}(\alpha')\vert \leq \varepsilon.$
Set  $$\begin{cases}\alpha_{\min}=\inf\left\{\alpha : \ D_{\pi_2(\mu)}(\alpha)>0 \right\} \\ \alpha_{\max}=\sup\left\{\alpha : \ D_{\pi_2(\mu)}(\alpha)>0 \right\}.\end{cases} $$
We recall that, since $\pi_2(\mu)$ is a non atomic self-similar measure, it is known that $0<\alpha_{\min}\leq \alpha_{\max}<+\infty.$ For $0\leq k\leq \lfloor \frac{\alpha_{\max}-\alpha_{\min}}{\varepsilon} \rfloor$, set $$ I_k =[\alpha_{\min}+k\varepsilon,\alpha_{\min}+(k+1)\varepsilon[.$$
We also write, for $k=0,...,\lfloor \frac{\alpha_{\max}-\alpha_{\min}}{\varepsilon} \rfloor,$ $\alpha_k =\alpha_{\min}+k\varepsilon$ and $\alpha_{\lfloor \frac{\alpha_{\max}-\alpha_{\min}}{\varepsilon} \rfloor+1}=\alpha_{\max}.$ 

Write also $\theta_0=-\frac{\log(p_0)}{\log b}$ and $\theta_1 =-\frac{\log(p_{b-1})}{\log b}.$ By changing slightly $\varepsilon$ if one must, we may assume that $\theta_{0},\theta_1 \notin \left\{\alpha_i\right\}_{1\leq k\leq \lfloor \frac{\alpha_{\max}-\alpha_{\min}}{\varepsilon} \rfloor}.$

The following argument follows from Theorem \ref{Thmlargedev} applied to $\pi_2(\mu_0)$ (which satisfies the open set condition).
\begin{lemme}
\label{LemmaMultifracForma}
Let $\alpha_{\min}\leq \alpha\leq  \alpha_{\max} $ be a real number and, given $r>0$ let us write $$\mathcal{P}_{\alpha}(r,\varepsilon)=\sup\#\left\{\mathcal{T} \right\},$$
where $\mathcal{T}$ is a $r$-packing of $\supp(\pi_2(\mu_0))$ with, for every $B \in\mathcal{T}$, $$\vert B \vert^{\alpha+\varepsilon} \leq \pi_2\mu_0(B)\leq \vert B \vert^{\alpha-\varepsilon}.$$ 
There exists $r_{\alpha}>0$ such that for every $r\leq r_{\alpha},$ one has $$r^{-D_{\pi_2(\mu_0)}(\alpha)+\varepsilon}\leq \mathcal{P}_{\alpha}(r,\varepsilon)\leq r^{-D_{\pi_2(\mu_0)}(\alpha)-\varepsilon} .$$
\end{lemme}
Let us fix first $N$ large enough so that for every $0\leq i\leq \lfloor \frac{\alpha_{\max}-\alpha_{\min}}{\varepsilon} \rfloor+1,$  
 $r_{\alpha_i}\geq 2^{-N+1}.$ In addition, for $n\in\mathbb{N},$ we fix a packing $\mathcal{P}_{i,n,\varepsilon}$ of $$E_{i,n,\varepsilon}:=\left\{x : \ \vert B(x,2^{-n}) \vert^{\alpha_i+\varepsilon} \leq \pi_2(\mu_0)(B(x,2^{-n}))\leq \vert B(x,2^{-n}) \vert^{\alpha_i-\varepsilon}\right\}$$
by balls centered on $E_{i,n,\varepsilon}$ and of radius $2^{-n-2}.$ Notice that $$E_{i,n,\varepsilon}\subset \bigcup_{B\in\mathcal{P}_{i,N,\varepsilon}}2B$$
and  that 
$$\bigcup_{0\leq i\leq \lfloor \frac{\alpha_{\max}-\alpha_{\min}}{\varepsilon} \rfloor+1 }E_{i,n,\varepsilon}=\pi_2(K).$$
Let $D$ be a $b$-adic cude with $\mu_0(D)= \vert D \vert^{s_0}$ and $2^{\frac{n-1}{\tau_1}}\leq k\leq 2^{\frac{n}{\tau_1}}$ such that $X_k \in D.$

In order to estimates $\mathcal{H}^{s}_{\infty}\Big(R_k \cap K\Big)$, for $s \geq 0,$ we will distinguish two cases:

\bigskip

 \textbf{Case 1: $X_k$ is near the boundary of $D$} 

Recall that for $i=1,0,$ $$\alpha_i = \lim_{r \to 0^+}\frac{\log \pi_2 \mu_0\Big(B(i,r)\Big)}{\log r}=\frac{-\log p_i}{\log b}$$
was assumed to satisfy $\alpha_0 \leq \alpha_1.$

Let $X_k$ be such that  
 $$X_k \in f_D\Big([0,1]\times[0,b^{-(\tau-1)n}]\cup [0,1]\times[1-b^{-(\tau-1)n},1]\Big).$$
 In this case, if $R_k:=R_{\tau_1,\tau_2}(X_k,\frac{1}{k})$ intersects $D$ and an other cube $D'$ of generation $n$, one has
\begin{align*}
R_k \subset S_{D,D'}:=&f_D\Big([0,1]\times[0,b^{-(\tau-1)(n-1)}]\cup [0,1]\times[1-b^{-(\tau-1)(n-1)},1]\Big)\\
&\bigcup f_{D'}\Big([0,1]\times[0,b^{-(\tau-1)(n-1)}]\cup [0,1]\times[1-b^{-(\tau-1)(n-1)},1]\Big).
\end{align*} 
so that $$\mathcal{H}^{s}_{\infty}(R_k \cap K)\leq b^{-(n-1)\max\left\{s,s\tau -(\tau-1)(s_0 -\alpha_0)\right\}}.$$
 And, since the coding of $0$ has only $0$'s as digits, by self-similarity, $$\mu_0\Big(S_{D,D'}\Big)\leq 2\times \mu_0(D)\times b^{-(\tau-1)(n-1)\alpha_0},$$
 one has $$ E[\#\left\{2^{\frac{n-1}{\tau_1}}\leq k\leq 2^{\frac{n}{\tau_1}}: \ X_k \in S_{D,D'}\right\}]\leq 2^{\frac{n}{\tau_1}} \times 2\mu_0(D)\times b^{-(\tau-1)(n-1)\alpha_0}.$$
 This yields that 
\begin{align*}
 E[\#\left\{2^{\frac{n-1}{\tau_1}}\leq k\leq 2^{\frac{n}{\tau_1}}: \ X_k \in \bigcup_{D,D'} S_{D,D'}\right\}]\leq C  \times b^{-n((\tau-1)\alpha_0 -\frac{1}{\tau_1})}.
\end{align*} 
 By Markov's inequality,  $$\mathbb{P}\Big(\#\left\{b^{\frac{n-1}{\tau_1}}\leq k\leq b^{\frac{n}{\tau_1}}: \ X_k \in \bigcup_{D,D'} S_{D,D'}\right\}\geq b^{n\varepsilon} b^{-n((\tau-1)\alpha_0 -\frac{1}{\tau_1})}\Big)\leq b^{-n\varepsilon},$$
 which, by Borel cantelli lemma, implies that almost surely, there exists $n$ large enough so that $$\#\left\{b^{\frac{n-1}{\tau_1}}\leq k\leq b^{\frac{n}{\tau_1}}: \ X_k \in \bigcup_{D,D'} S_{D,D'}\right\}\leq  b^{-n((\tau-1)\alpha_0 -\frac{1}{\tau_1}-\varepsilon)}.$$
 
Notice that if  $(\tau-1)\alpha_0 -\frac{1}{\tau_1}-\varepsilon>0,$ then almost surely, for $n$ large enough, $$\#\left\{b^{\frac{n-1}{\tau_1}}\leq k\leq b^{\frac{n}{\tau_1}}: \ X_k \in \bigcup_{D,D'} S_{D,D'}\right\}=0,$$
so that we may assume that $(\tau-1)\alpha_0 -\frac{1}{\tau_1}-\varepsilon\leq 0,$
 
 Thus if $s$ is solution to $\max\left\{s,s\tau -(\tau-1)(s_0 -\alpha_0)\right\}=-(\tau-1)\alpha_0 +\frac{1}{\tau_1}+\varepsilon,$ one has for every $\kappa>1,$
 $$\sum_{X_k \in \bigcup_{D,D'} S_{D,D'}}\mathcal{H}^{s\kappa}_{\infty}\Big(R_k\cap K\Big)\leq \sum_{n\geq 1}b^{-n((\tau-1)\alpha_0 -\frac{1}{\tau}-\varepsilon)}b^{n\kappa((\tau-1)\alpha_0 -\frac{1}{\tau}-\varepsilon)) }<+\infty.$$ 
 We conclude that $$\dim_H \limsup_{X_k \in \bigcup_{D,D'} S_{D,D'}}R_k \leq \kappa s $$
for any $\kappa,$ so that  $$\dim_H \limsup_{X_k \in \bigcup_{D,D'} S_{D,D'}}R_k \leq  s.$$
In addition, recall that for any $\alpha \in \mbox{Spectr}(\pi_2 \mu_0)$, $\alpha-D_{\pi_2 \mu_0} \geq 0$ and that $s_{\alpha_0}$ is solution to the equation $$\max\left\{s,s\tau -(\tau-1)(s_0 -\alpha_0)\right\}=-(\tau-1)(\alpha_0-D_{\pi_{2}(\mu_0)}(\alpha_0)) +\frac{1}{\tau_1}.$$
This yields $s \leq s_{\alpha_0}+\varepsilon$ provided that $\varepsilon$ was chosen small enough to begin with.

\bigskip

\textbf{Case 2: $X_k$ is not near the boundary of $D$}

Assume now that $X_k \in D \setminus  f_D\Big([0,1]\times[0,b^{(\tau-1)n}]\cup [0,1]\times[1-b^{(\tau-1)n},1]\Big)$ and recall that, by self-similarity, one has $$\mathcal{H}^{s}_{\infty}(R\cap K \cap D)\leq \mathcal{H}^{s}_{\infty}(R\cap K)\leq 3\mathcal{H}^{s}_{\infty}(R\cap K \cap D).$$
Recall that

 $$D \cap K \subset \bigcup_{0\leq i\leq  \frac{\alpha_{\max}-\alpha_{\min}}{\varepsilon} \rfloor+1}f_{D}\circ \pi_{2}^{-1}\Big( \bigcup_{B \in \mathcal{P}_{i,\lfloor(\tau-1)n\rfloor -1,\varepsilon}}B\Big) .$$

Fix $0\leq i\leq \lfloor \frac{\alpha_{\max}-\alpha_{\min}}{\varepsilon} \rfloor+1$ and assume that $X_k \in  f_{D}\circ \pi_2^{-1}(B),$ where $B \in \mathcal{P}_{i,\lfloor(\tau-1)n\rfloor -1,\varepsilon}.$ By definition of $\mathcal{P}_{i,\lfloor (\tau-1)n\rfloor-1,\varepsilon}$ and by self-similarity, one has $$\mu_0(f_{D}\circ \pi_2^{-1}(2B))\leq (2\vert B\vert)^{\alpha_i-\varepsilon}\times \mu_0(D).$$
In addition, if $C$ is a $b$-adic cube of generation $\lfloor n \tau \rfloor$ intersecting $R\cap D \cap K,$ then $C \subset f_{D}\circ \pi_2^{-1}(2B).$ We conclude that 

\begin{align*}
\#\left\{C\in \mathcal{D}_{\lfloor n \tau \rfloor} : C\cap R\cap D \cap K \neq \emptyset\right\}&\leq  \frac{\mu_0\Big(f_{D}\circ \pi_2^{-1}(2B)\Big)}{b^{-\lfloor n \tau \rfloor s_0}}\leq C b^{n\tau s_0-n s_0 -(\tau-1)(\alpha_i -\varepsilon)}\\
&\leq C b^{n(\tau-1)(s_0 -\alpha_i +\varepsilon)}.
\end{align*}

By covering $R\cap D \cap K$ by $b$-adic cubes of generation $\lfloor n \tau \rfloor$ and $n$, we obtain that $$\mathcal{H}^{s}_{\infty}(R\cap D \cap K)\leq b^{-n\max\left\{s,s\tau-(\tau-1)(s_0-\alpha_i+\varepsilon)\right\}}.$$

Moreover, by self-similarity, 

\begin{align*}
&\mu_0\Big( f_{D}\circ \pi_{2}^{-1}\Big( \bigcup_{B \in \mathcal{P}_{i,\lfloor(\tau-1)n\rfloor -1,\varepsilon}}B\Big)\Big)\\
&\leq \mu_0(D)\times \#\left\{\mathcal{P}_{i,\lfloor(\tau-1)n\rfloor -1,\varepsilon}\right\} \times \Big(2 \vert B\vert\Big)^{\alpha_i -\varepsilon}\leq C \mu_0(D)\times b^{-n((\tau-1)(\alpha_i -\varepsilon)-(\tau-1)(D_{\pi_2(\mu_0)}+\varepsilon))}\\
&\leq C\mu_0(D)b^{-(\tau-1)n(\alpha_i -D_{\pi_2(\mu_0)}(\alpha_i)-2\varepsilon)}.
\end{align*}
Writing $$\begin{cases}G_{n,i,\varepsilon}= \bigcup_{D \in \mathcal{D}_n} f_{D}\circ \pi_{2}^{-1}\Big( \bigcup_{B \in \mathcal{P}_{i,\lfloor(\tau-1)n\rfloor -1,\varepsilon}}B\Big)\setminus  f_D\Big([0,1]\times[0,b^{(\tau-1)n}]\cup [0,1]\times[1-b^{(\tau-1)n},1]\Big) \\

U_{n,i,\varepsilon}=\left\{b^{\frac{(n-1)}{\tau_1}}\leq k\leq b^{\frac{n}{\tau_1}} : \ X_k \in G_{n,i,\varepsilon}\right\} \end{cases} $$
\begin{align*}
E[\#\left\{U_{n,i,\varepsilon} \right\}]\leq b^{\frac{n}{\tau_1}}\times Cb^{-(\tau-1)n(\alpha_i -D_{\pi_2(\mu_0)}(\alpha_i)-2\varepsilon)}=C b^{-n(-\frac{1}{\tau_1}+(\tau-1)(\alpha_i -D_{\pi_{2}(\mu_0)}(\alpha_i)-2\varepsilon))}.
\end{align*}
Thus, by Markov's inequality and Borel-Cantelli lemma, almost surely,  there exists $n_{i,\varepsilon}$ large enough so that $$\#\left\{U_{n,i,\varepsilon}\right\}\leq  C b^{-n(-\frac{1}{\tau_1}+(\tau-1)(\alpha_i -D_{\pi_{2}(\mu_0)}(\alpha_i)-3\varepsilon))}.$$
Using the same argument as in case $(1)$, one assume that $\alpha_i$ is such that $$-\frac{1}{\tau_1}+(\tau-1)(\alpha_i -D_{\pi_{2}(\mu_0)}(\alpha_i)-3\varepsilon)\leq 0.$$

Let $s$ be the solution to $$\max\left\{s,s\tau-(\tau-1)(s_0-\alpha_i+\varepsilon)\right\}= \frac{1}{\tau_1}-(\tau-1)(\alpha_i -D_{\pi_{2}(\mu_0)}(\alpha_i)-3\varepsilon).$$
The same argument in case $(1)$ yields that for every $\kappa>1$ $$\sum_{n\geq 1 , k\in U_{n,i,\varepsilon}}\mathcal{H}^{\kappa s}_{\infty}(R_k \cap K)<+\infty,$$
so that $$\dim_H \limsup_{ k\in U_{n,i,\varepsilon}}R_k \leq s.$$
In addition, since $s_{\alpha_i}$ is solution to $$ $$ $$\max\left\{s,s\tau-(\tau-1)(s_0-\alpha_i)\right\}= \frac{1}{\tau_1}-(\tau-1)(\alpha_i -D_{\pi_{2}(\mu_0)}(\alpha_i)),$$
one has $s\leq s_{\alpha_i}+\varepsilon \leq \sup_{\alpha}s_{\alpha}+\varepsilon$ provided that $\varepsilon$ was chosen small enough to begin with.

\subsubsection{Proof of the lower-bound}

 Recall that, for $\alpha \in \mbox{Spectr}(\pi_2 \mu_0),$ $$s_{\alpha}=\min\left\{\frac{1-(\tau_2-\tau_1)(\alpha-D_{\pi_2 \mu_0}(\alpha))}{\tau_1},\frac{1+(\tau_2-\tau_1)(s_0 -2\alpha +D_{\pi_2 \mu_0}(\alpha))}{\tau_2}\right\}.$$

The following fact is a consequence of Proposition \ref{propriOSC}

\begin{proposition}
\label{propoMesauxi}
There exists $\nu_{\alpha}$, $\times b$-ergodic (hence exact-dimensional) such that
\begin{itemize}
\item[(1)]$\nu_{\alpha}\Big(\left\{x : \lim_{r\to 0^+}\frac{-\log \pi_2 \mu_0 (D_n(x))}{n\log b}=\alpha\right\}\Big)=1,$\medskip
\item[(2)] $\dim_H \nu_{\alpha}=D_{\pi_2 \mu_0}(\alpha).$
\end{itemize}

\end{proposition}
By Proposition \ref{propoMesauxi}, there exists $N\in\mathbb{N}$ such that the set

 $$F_{N,\varepsilon,\alpha}=\left\{x : \ \forall n\geq N, \begin{cases}  b^{-n(D_{\pi_2 \mu_0}(\alpha)+\varepsilon)}\leq\nu_{\alpha}(D_n(x))\leq b^{-n(D_{\pi_2 \mu_0}(\alpha)-\varepsilon)}\\ b^{-n(\alpha+\varepsilon)}\leq \pi_2\mu_0(D_n(x))\leq b^{-n(\alpha-\varepsilon)} \end{cases} \right\}$$
satisfies $$\nu_{\alpha}(F_{N,\varepsilon,\alpha})\geq \frac{1}{2}.$$
Notice that, for any $n\in\mathbb{N}$ and $x \in \supp(\nu_{\alpha})\subset \left\{0\right\}\times [0,1],$ $$\pi_2^{-1}(D_n(x))\cap K =[0,1]\times I \cap K,$$
where $\left\{0\right\}\times I=D_n(x) \cap  \left\{0\right\}\times [0,1].$ Hence, $$C^{-1}\leq\frac{\mathcal{H}^{s}_{\infty}([0,1]\times I \cap K)}{\inf_{0\leq \gamma \leq 1} b^{\gamma n s_0}\times \pi_2\mu_0(D_{\lfloor\gamma n \rfloor}(x))b^{-n\gamma s}}\leq C.$$
Thus, for any $x\in F_{N,\alpha,\varepsilon},$ for any $n\geq N$ and $\gamma \geq \frac{n}{N},$ one has $$b^{-\gamma n(\alpha+\varepsilon)} \leq\pi_2\mu_0(D_{\lfloor\gamma n \rfloor}(x)) \leq b^{-\gamma n(\alpha-\varepsilon)}$$
and there exists $\kappa_N >0$ such that $$\kappa_{N}^{-1}\leq\inf_{\frac{n}{N}\leq \gamma \leq 1} b^{-\gamma n s_0}\times \pi_2\mu_0(D_{\lfloor\gamma n \rfloor}(x))b^{-n\gamma s} \leq \kappa_N.$$
This yields 
\begin{align*}
\inf_{0\leq \gamma \leq \frac{n}{N}}  b^{-\gamma n(s-s_0+\alpha+\varepsilon)}\leq\inf_{0\leq \gamma \leq \frac{n}{N}} b^{\gamma n s_0}\times \pi_2\mu_0(D_{\lfloor\gamma n \rfloor}(x))b^{-n\gamma s} \leq \inf_{0\leq \gamma \leq \frac{n}{N}}  b^{-\gamma n(s-s_0+\alpha-\varepsilon)}
\end{align*}
If $s \geq s_0 -\alpha +\varepsilon$ or $s\leq s_0 -\alpha -\varepsilon,$ both $\gamma\mapsto\gamma(s-(s_0-\alpha-\varepsilon))$ and $\gamma\mapsto\gamma(s-(s_0-\alpha+\varepsilon))$ are non increasing or non decreasing. We conclude that, for any such $s$, $$ \frac{1}{C\log n}\min\left\{\kappa_N, b^{- n(s-s_0+\alpha+\varepsilon)}\right\} \leq\mathcal{H}^{s}_{\infty}([0,1]\times I \cap K)\leq C\log(n) \min\left\{\kappa_N,b^{- n(s-s_0+\alpha-\varepsilon)}\right\} .$$ 
Moreover, since the content is non increasing in $s$, for any $s_0 -\alpha -\varepsilon \leq s\leq s_0 -\alpha +\varepsilon,$ one has $$ \frac{1}{C\log n}\min\left\{\kappa_N, b^{- n(s-s_0+\alpha+2\varepsilon)}\right\} \leq\mathcal{H}^{s}_{\infty}([0,1]\times I \cap K)\leq C\log(n) \min\left\{\kappa_N,b^{- n(s-s_0+\alpha-2\varepsilon)}\right\} $$ 
so that the latest inequalities holds for any $s\leq s_0.$
From now on, we consider $n$ large enough so that $(\tau-1)n \geq N+1$ and $$\mathcal{P}_{(\tau-1)n}=\left\{D_{\lfloor (\tau-1)n \rfloor +1}(x),x\in F_{N,\varepsilon,\alpha} \right\}.$$
It follows from the definition of $F_{N,\varepsilon,\alpha}$ that

$$\begin{cases}\#\mathcal{P}_{(\tau-1)n} \geq b^{(\tau-1)n(D_{\pi_2 \mu_0}(\alpha)-\varepsilon)} \\ \forall I \in \mathcal{P}_{(\tau-1)n}, \ b^{-n(\tau-1)(\alpha+\varepsilon)}\leq \pi_2\mu_0(D_n(x))\leq b^{-n(\tau-1)(\alpha-\varepsilon)}.\end{cases} $$
Hence, we obtain $$\pi_2 \mu_0\Big(\bigcup_{I \in \mathcal{P}_{(\tau-1)n} }I\Big)\geq b^{-n(\tau-1)(\alpha -D_{\pi_2 \mu_0}(\alpha)+2\varepsilon)}.$$

Let us write $$\delta_{\alpha}=\frac{s_0}{\frac{1}{\tau_1}-(\tau -1)(\alpha-D_{\pi_2 \mu_0}(\alpha))} .$$
Notice that, since $\tau_1 \geq\frac{1}{s_0}$, $\tau\geq 1$ and $\alpha-D_{\pi_2 \mu_0}(\alpha)\geq 0,$ one has $\delta_{\alpha}\geq 1.$ 

Fix $x\in K$ and set $k=\lfloor \frac{n}{\delta_{\alpha}+\varepsilon} \rfloor$ For any $D\subset D_{k}(x)$ with $D \in \mathcal{D}_{n},$   by self-similarity, one has  $$\mu_0 \Big(f_D\Big(\pi_2^{-1}\Big(\bigcup_{I \in \mathcal{P}_{(\tau-1)n} }I\Big)\Big)\Big)=\mu_0(D)\times \pi_2 \mu_0\Big(\bigcup_{I \in \mathcal{P}_{(\tau-1)n} }I\Big).$$
This yields that 
\begin{align*}
&\mu_0\Big(\bigcup_{D\in\mathcal{D}_n, D\subset D_{k}(x) }(f_D\Big(\pi_2^{-1}\Big(\bigcup_{I \in \mathcal{P}_{(\tau-1)n} }I\Big)\Big)\Big)=\mu(D_{k}(x))\times  \pi_2 \mu_0\Big(\bigcup_{I \in \mathcal{P}_{(\tau-1)n} }I\Big)\\
&\geq b^{-n\Big(\frac{s_0}{(1+\varepsilon)\delta_{\alpha}}+(\tau-1)(\alpha -D_{\pi_2 \mu_0}(\alpha)+2\varepsilon)\Big)}=b^{-n(\frac{1}{(1+\varepsilon)\tau_1})+(\tau-1)((1-\frac{1}{1+\varepsilon})(\alpha -D_{\pi_2 \mu_0}(\alpha))+2\varepsilon)}\geq b^{\frac{-n}{(1+\frac{\varepsilon}{2})\tau_1}}
\end{align*}
provided that $\varepsilon$ was chosen small enough to begin with. Write $$A_k =\bigcup_{D\in\mathcal{D}_n, D\subset D_{k}(x) }(f_D\Big(\pi_2^{-1}\Big(\bigcup_{I \in \mathcal{P}_{(\tau-1)n} }I\Big)\Big).$$

Then, for any $p\in\mathbb{N}$ and any constant $C>0$ and $q\geq Cb^{\frac{n}{\tau_1}},$

\begin{equation}
\label{EquaInde}
\mathbb{P}\Big(X_p,...,X_{p+q}\notin A_k\Big)\leq (1-\mu_0(A_k))^{Cb^{\frac{n}{\tau_1}}}\leq e^{-Cb^{\frac{n}{\tau_1}}\mu_0(A_k)}.
\end{equation}
Since $b^{\frac{n}{\tau_1}}\mu_0(A_k)\geq b^{\frac{n}{\tau_1}(1-\frac{1}{1+\varepsilon})}\to +\infty.$  Since $\#\left\{b^{\frac{n}{\tau_1}}<p \leq b^{\frac{n+1}{\tau_1}}\right\}\geq Cb^{\frac{n}{\tau_1}},$ by Borel-Cantelli we obtain that almost surely, for every large enough $k$, there exists $b^{\frac{n}{\tau_1}}\leq p \leq b^{\frac{n+1}{\tau_1}}$ such that $X_p \in A_k.$ Now, fixing $D\in\mathcal{D}_n$ such that $D\subset D_{k}(x),$ $I\in\mathcal{P}_{(\tau-1)n}$ and $X_p \in f_D(\pi_2^{-1}(I)).$ Notice that $f_D\Big(\pi_2^{-1}(I)\cap [0,1]^2\Big)$ is a rectangle, product of a $b$-adic interval of generation $n$ with a $b$-adic interval of generation $\tau n.$ Since $p\leq 2^{\frac{n-1}{\tau_1}},$ one has $$\begin{cases} \frac{1}{k^{\tau_1}}\geq 2^{-n} \\ \frac{1}{k^{\tau_2}}\geq 2^{-n \tau}\end{cases}$$
which yields that $f_D\Big(\pi_2^{-1}(I)\cap [0,1]^2\Big)\subset R_k$ and, provided that $n$ is large enough,
\begin{align*}
\mathcal{H}^{s}_{\infty}\Big(R_k \cap K\Big)&\geq \mathcal{H}^{s}_{\infty}\Big(f_D\Big(\pi_2^{-1}(I)\cap [0,1]^2\Big)\Big)\\
&\geq \frac{C}{\log((\tau-1)n)}\vert D \vert^s \min\left\{\kappa_N ,b^{-n(\tau-1)(s-s_0 +2\varepsilon)}\right\} \\
&\geq \theta_N \min\left\{b^{-n(1+\varepsilon)s},b^{-n(\tau-1)(s-s_0 +3\varepsilon)}\right\}\\
&=\theta_N b^{-n \max\left\{s(1+\varepsilon),\tau s-(\tau-1)(s_0 -\alpha- 3\varepsilon)\right\}}.
\end{align*}
Also, $$s_0 k =\frac{s_0 n}{(1+\varepsilon)\delta_{\alpha}}=\frac{1}{1+\varepsilon}\times\Big(\frac{1}{\tau_1}-(\tau-1)(\alpha -D_{\pi_2(\mu_0)}(\alpha))\Big)$$
so that, if $s$ is solution to 

\begin{equation}
\label{equas}
\frac{1}{1+\varepsilon}\times\Big(\frac{1}{\tau_1}-(\tau-1)(\alpha -D_{\pi_2(\mu_0)}(\alpha))\Big)=  \max\left\{s(1+\varepsilon),\tau s-(\tau-1)(s_0 -\alpha- 3\varepsilon)\right\},
\end{equation}
for every $x$, almost surely, for every large enough $k\in \mathbb{N}$, there exists $p$ such that $X_p \in D_k(x)$ and 
$$\mathcal{H}^{s}_{\infty}(R_p \cap K)\geq \mu_0(D_{k}(x). $$

Note also that $$D_k(x) \cup R_p \subset B\Big(X_p,\frac{1}{p^{\frac{\tau_1}{(1+\varepsilon)\delta_{\alpha}}}}\Big)\text{ and } \frac{1}{p^{\frac{\tau_1}{(1+\varepsilon)\delta_{\alpha}}}}\leq b^{2}\vert D_n(x)\vert$$
which implies in particular that $x\in B\Big(X_p,\frac{1}{p^{\frac{\tau_1}{(1+\varepsilon)\delta_{\alpha}}}}\Big) $ and  $$\mathcal{H}^{s}_{\infty}(R_p \cap K)\geq C\mu_0\Big(B\Big(X_p,\frac{1}{p^{\frac{\tau_1}{(1+\varepsilon)\delta_{\alpha}}}}\Big)\Big)$$
  so that, almost surely, $$x \in \limsup_{ p : \ \mathcal{H}^{s}_{\infty}(R_p \cap K)\geq C\mu_0\Big(B\Big(X_p,\frac{1}{p^{\frac{\tau_1}{(1+\varepsilon)\delta_{\alpha}}}}\Big)\Big)}B\Big(X_p,\frac{1}{p^{\frac{\tau_1}{(1+\varepsilon)\delta_{\alpha}}}}\Big).$$
  Thus, by Fubini,
\begin{align*}
&\int \mathbb{P}\Big(\limsup_{ p : \ \mathcal{H}^{s}_{\infty}(R_p \cap K)\geq C\mu_0\Big(B\Big(X_p,\frac{1}{p^{\frac{\tau_1}{(1+\varepsilon)\delta_{\alpha}}}}\Big)\Big)}B\Big(X_p,\frac{1}{p^{\frac{\tau_1}{(1+\varepsilon)\delta_{\alpha}}}}\Big)\Big)d\mu_0(x)=1 \\
&\Leftrightarrow \int  \mu_0\Big(\limsup_{ p : \ \mathcal{H}^{s}_{\infty}(R_p \cap K)\geq C\mu_0\Big(B\Big(X_p,\frac{1}{p^{\frac{\tau_1}{(1+\varepsilon)\delta_{\alpha}}}}\Big)\Big)}B\Big(X_p,\frac{1}{p^{\frac{\tau_1}{(1+\varepsilon)\delta_{\alpha}}}}\Big)\Big)d\mathbb{P}=1
\end{align*}
and almost surely, $$\mu_0\Big(\limsup_{ p : \ \mathcal{H}^{s}_{\infty}(R_p \cap K)\geq C\mu_0\Big(B\Big(X_p,\frac{1}{p^{\frac{\tau_1}{(1+\varepsilon)\delta_{\alpha}}}}\Big)\Big)}B\Big(X_p,\frac{1}{p^{\frac{\tau_1}{(1+\varepsilon)\delta_{\alpha}}}}\Big)\Big)=1.$$
Theorem \ref{MTPss} yields $$\dim_H \limsup_{k\to +\infty}R_k \geq s.$$
Recalling that $s$ satisfies \eqref{equas} and that $s_{\alpha}$ is solution to the equation $$\Big(\frac{1}{\tau_1}-(\tau-1)(\alpha -D_{\pi_2(\mu_0)}(\alpha))\Big)=  \max\left\{s,\tau s-(\tau-1)(s_0 -\alpha)\right\},$$
one has $s\geq s_{\alpha}-\varepsilon$ provided that $\varepsilon$ was chosen small enough. Since, $\varepsilon, \alpha$ are arbitrary, Theorem \ref{ThmRand} is proved.

\subsubsection{Dealing with the case of orbits}
\label{SectionDyna}
In this section we explain how to adapt the proof made in the case where $(X_n)_{n\in\mathbb{N}}$ is i.i.d. of law $\mu_0$ to the case where $X_n =T_b^{n}(x),$ where $x$ is $\mu_0$-typical and $T_b: \mathbb{T}^2 \to \mathbb{T}^2$ is defined by $T_b(x)=bx$. 

First, notice that the independency of the sequence $(X_n)_{n\in\mathbb{N}}$ was not used to establish the upper-bound. Thus, the argument readily applies to $(Y_n =b^n x)_{n\in\mathbb{N}}.$

However, the independency is used while estimating the lower-bound in \eqref{EquaInde}. With the same notations as in the corresponding subsection, the idenpendency was used  to establish that given any $b$-adic cube $D \in\mathcal{D}_{\lfloor \frac{n}{(1+\varepsilon)\delta_{\alpha}} \rfloor}$ with $D\cap K \neq \emptyset$ if $$\mathcal{R}_D=\left\{D'\in\mathcal{D}_{n}, D'\cap K \neq \emptyset f_{D'}\Big(\pi_2^{-1}(B)\Big), B\in \mathcal{P}_{\lfloor (\tau-1)n \rfloor} \right\},$$
then almost surely there exists $b^{\frac{n-1}{\tau_1}}\leq k \leq b^{\frac{n}{\tau_1}}$ such that $X_k \in \bigcup_{R \in\mathcal{R}_D}R.$

It is established in \cite[Proposition 3.10]{Edergo} that there exists $C>0$ and $\tau_1$ such that for every ball $A$ and Borel set $B$, for every $n\in\mathbb{N},$ one has  $$\mu_0\Big(T_b^{-n}(B)\cap A\Big) \leq \mu(A)\times \mu(B) +C\tau^{n}\mu(B).$$
Thus, as $\#\left\{R_D\right\} $ is polynomial in $\vert D \vert,$ the same argument as in \cite[Lemma 7.6]{Edergo} holds to show that the same conclusion as in the case where $(X_n)_{n\in\mathbb{N}}$ was an i.i.d. sequence holds true.

\bigskip

\section{Acknowledgments}

The author would like to thank   D.J. Feng and S.Seuret for sharing the random covering problem of self-similar sets by rectangles with him.

\bibliographystyle{plain}
\bibliography{bibliogenubi}

\begin{thebibliography}{10}

\bibitem{AllenB}
D.~Allen and S.~Bárány.
\newblock On the {H}ausdorff measure of shrinking target sets on self-conformal
  sets.
\newblock {\em Mathematika}, 67(4):807--839, 2021.

\bibitem{AJW}
D.~Allen, T~Jordan, and B.~Ward.
\newblock Rectangular shrinking targets on self-similar carpets.
\newblock {\em arXiv:2510.04875}, 2025.

\bibitem{Baker}
S.~Baker.
\newblock Intrinsic {D}iophantine approximation for overlapping iterated
  function systems.
\newblock {\em Math.Ann.}, 338:3259–3297, 2024.

\bibitem{RamsBara}
B.~Barany and M.~Rams.
\newblock Shrinking targets on bedford–mcmullen carpets.
\newblock {\em Proceeding of the London mathematical Society}, 117:951--995,
  2018.

\bibitem{BS2}
J.~Barral and S.~Seuret.
\newblock Ubiquity and large intersections properties under digit frequencies
  constraints.
\newblock {\em Math. Proc. Camb. Phil. Soc.}, 145(3):527--548, 2008.

\bibitem{ed2}
E.~Daviaud.
\newblock Extraction of optimal sub-sequences of balls and application to
  optimality estimates of mass transference principles.
\newblock {\em arXiv:2204.01304}, 2022.

\bibitem{ED4}
E.~Daviaud.
\newblock Dynamical coverings for ${C}^{1}$ weakly conformal {IFS} with
  overlaps.
\newblock {\em Ergodic theory and dynamical systems}, 2023.

\bibitem{ED1}
E.~Daviaud.
\newblock A dimensional mass transference principle from ball to rectangles for
  projections of gibbs measures and applications.
\newblock {\em J.M.A.A.}, 538,issue 1, 2024.

\bibitem{ED3}
E.~Daviaud.
\newblock A dimensional mass transference principle for {B}orel probability
  measures and applications.
\newblock {\em Advances in mathematics}, 474, 2025.

\bibitem{Edergo}
E.~Daviaud.
\newblock Hausdorff dimension of dynamical diophantine approximation under
  mixing properties.
\newblock {\em arXiv:2502.13051}, 2025.

\bibitem{EP}
F.~Ekström and T.~Persson.
\newblock Hausdorff dimension of random limsup sets.
\newblock {\em Journal of the London Mathematical Society}, 98:661–686, 2018.

\bibitem{Fa1}
K.~Falconer.
\newblock {\em Fractal Geometry}.
\newblock John Wiley \& Sons, 1990.

\bibitem{F}
K.~Falconer.
\newblock {\em Fractal geometry}.
\newblock John Wiley \& Sons, Inc., Hoboken, NJ, second edition, 2003.
\newblock Mathematical foundations and applications.

\bibitem{FengJ2Suo}
A.-H. Feng, E.~Jarvenp\"aa, M.~Jarvenp\"aa, and V.~Suomala.
\newblock Dimensions of random covering sets in {R}iemann manifolds.
\newblock {\em J. London Math. Soc.}, 6:229--244.

\bibitem{FH}
D.~Feng and H.~Hu.
\newblock Dimension theory of iterated function systems.
\newblock {\em Comm. Pure Appl. Math.}, 62:1435--1500, 2009.

\bibitem{Galadim}
S.~Galatolo.
\newblock Dimension via waiting time and recurrence.
\newblock {\em I.M.R.N.}, 12, 2005.

\bibitem{HV}
R.~Hill and S.~Velani.
\newblock The ergodic theory of shrinking targets.
\newblock {\em Inv. Math.}, 119:175--198, 1995.

\bibitem{Hutchinson}
J.E. Hutchinson.
\newblock Fractals and self similarity.
\newblock {\em Indiana Univ. Math. J.}, 30:713--747, 1981.

\bibitem{JaffRiemann}
S.~Jaffard.
\newblock The spectrum of singularities of {R}iemann's function.
\newblock {\em Rev. Mat. Iberoamericana}, 12(2):441--460, 1996.

\bibitem{JJMS}
E.~Järvenpää, M~Järvenpää, M.~Myllyoja, and Ö~Stenflo.
\newblock The {E}kström-persson conjecture regarding random covering sets.
\newblock {\em Journal of the London Mathematical Society}, 111, issue 1, 2024.

\bibitem{JarvSeur}
E.~Järvenpää, M~Myllyoja, and S.~Seuret.
\newblock Hitting probabilites, and the ekström-persson conjecture.
\newblock {\em availaible on arxiv}, 2025.

\end{thebibliography}

\end{document}